\documentclass[11pt,a4paper]{amsart} 
\usepackage[T1]{fontenc}
\usepackage[latin9]{inputenc}
\usepackage[a4paper]{geometry}
\usepackage{babel}
\usepackage{mathtools}
\usepackage{amsthm}
\usepackage{amsmath}
\usepackage{amssymb}
\usepackage[unicode=true,pdfusetitle,
 bookmarks=true,bookmarksnumbered=false,bookmarksopen=false,
 breaklinks=false,pdfborder={0 0 1},backref=false,colorlinks=false]
 {hyperref}
\hypersetup{
 pdftex,pagebackref,colorlinks,linktocpage,linkcolor=blue,citecolor=blue}

\usepackage{color} 
\usepackage{mathrsfs} 
\usepackage{mathtools}\mathtoolsset{showonlyrefs}
\usepackage{comment} 
\usepackage{hyperref} 
\numberwithin{equation}{section}
\makeatletter
\makeatletter
\theoremstyle{plain}
\ifx\thesection\undefined
\newtheorem{thm}{\protect\theoremname}
\else
\newtheorem{thm}{\protect\theoremname}[section]
\fi
\theoremstyle{remark}
\newtheorem{rem}[thm]{\protect\remarkname}
\theoremstyle{definition}
\newtheorem{defn}[thm]{\protect\definitionname}
\theoremstyle{plain}
\newtheorem{lem}[thm]{\protect\lemmaname}
\ifx\proof\undefined
\newenvironment{proof}[1][\protect\proofname]{\par
\normalfont\topsep6\p@\@plus6\p@\relax
\trivlist
\itemindent\parindent
\item[\hskip\labelsep\scshape #1]\ignorespaces
}{%
\endtrivlist\@endpefalse
}
\providecommand{\proofname}{Proof}
\fi
\theoremstyle{plain}
\newtheorem{cor}[thm]{\protect\corollaryname}
\theoremstyle{plain}
\newtheorem{prop}[thm]{\protect\propositionname}

\usepackage{fancyhdr}\@ifundefined{definecolor}{\usepackage{color}}{}
\usepackage{chngcntr}
\usepackage{layout}
\counterwithout{footnote}{section}
\usepackage{amsfonts}
\usepackage[font=small,labelfont=bf]{caption}

\allowdisplaybreaks[4]

\makeatother

\providecommand{\corollaryname}{Corollary}
\providecommand{\definitionname}{Definition}
\providecommand{\lemmaname}{Lemma}
\providecommand{\propositionname}{Proposition}
\providecommand{\remarkname}{Remark}
\providecommand{\theoremname}{Theorem}

\begin{document}

\title{Branching Particle Systems with Mutually Catalytic
Interactions}
	\author[A. Jamchi Fugenfirov]{Alexandra Jamchi Fugenfirov}
	\address[A. Jamchi Fugenfirov]{{\tt Faculty of Data and Decision Sciences,
Technion --- Israel Institute of Technology,
Haifa 3200003, Israel}}
	\email{{\tt f8.sasha@gmail.com}}
	\author[L. Mytnik]{Leonid Mytnik}
	\address[L. Mytnik]{{\tt Faculty of Data and Decision Sciences,
Technion --- Israel Institute of Technology,
Haifa 3200003, Israel}}
	\email{{\tt leonidm@technion.ac.il}}

\date{October 2023}

\cleardoublepage{}

\maketitle

\begin{abstract}

We study a continuous time Mutually Catalytic Branching model on the
$\mathbb{Z}^{d}$. The model describes the behavior of two different
populations of particles, performing random walk on the lattice in
the presence of branching, that is, each particle dies at a certain
rate and is replaced by a random number of offspring. The branching
rate of a particle in one population is proportional to the number
of particles of another population at the same site. We study the
long time behavior for this model, in particular, coexistence and
non-coexistence of two populations in the long run. Finally, we construct
a sequence of renormalized processes and use duality techniques to
investigate its limiting behavior. 

\end{abstract}



\section{Introduction}\label{sec:Introduction}

\subsection{Background and motivation}\label{sub:Background-and-Motivation}

In the last four decades there has been a lot of interest in spatial
branching models. These models include branching random walks, branching
Brownian motion, super-processes and so on. During the last three decades
branching models with interactions were studied very extensively on
the level of continuous state models and particle models. Below we
will give a partial list of branching models with interactions that
were studied in the literature. 

Models with catalytic branching, where one population catalyzes another
were studied in~\cite{DawsonFleischmann00}, \cite{LiMa08}, \cite{DawsonFleischmannMueller00}.
For measure-valued diffusions with mutual catalytic branching, see
Dawson-Perkins \cite{DawsonPerkins98}, Dawson et al.~\cite{DawsonAtEl02,MR1959845,EJP114}.
Models with symbiotic branching --- these are models with a correlation
in branching laws of two populations ---  were  investigated in~\cite{EtheridgeFleischmann98}, \cite{BlathDoringEtheridge11}, \cite{bib:hov18},  \cite{bib:bo21}, 
\cite{bib:bho16a}, \cite{bib:gm23}, \cite{bib:AM25} among others. 
Infinite rate branching models was introduced in \cite{KlenkeMytnik10},
\cite{KlenkeOeler10} and studied later in~\cite{bib:km12a}, \cite{bib:dkm17}, \cite{bib:dm12a}, \cite{bib:dm:13a}, \cite{bib:km20}. In addition, various particle
models were introduced in~\cite{Birkner2003}, \cite{KestenSidoravicius03}. 

Let us say a few words about mutually catalytic branching model in
the continuous state setting introduced in \cite{DawsonPerkins98}.


Dawson and Perkins in \cite{DawsonPerkins98} constructed the model
with $\mathbb{Z^{\textrm{\ensuremath{d}}}}$ being a space of sites
and $(u,v)\in\mathbb{R}_{+}^{\mathbb{Z}^{d}}\times\mathbb{R}_{+}^{\mathbb{Z}^{d}}$
a pair which undergo random migration and continuous state mutually
catalytic branching. The random migration is described by a $\mathbb{Z}^{d}$-valued
Markov chain with the associated $Q$-matrix, $\mbox{\ensuremath{Q=(q_{ij})}}$, subject to certain   technical conditions on $Q$ ad associated transition probabilities (see page~1090 
in~\cite{DawsonPerkins98}, and in particular ($H_0$), ($H_1$), ($H_2$) there).  
The branching rate of one population at a site is proportional to
the mass of the other population at the site. The system is modeled
by the following infinite system of stochastic differential equations:
\begin{equation}
\begin{cases}
\begin{array}{c}
u_{t}(x)=u_{0}(x)+\intop_{0}^{t}u_{s}Q(x)ds+\intop_{0}^{t}\sqrt{\widetilde{\gamma}u_{s}(x)v_{s}(x)}dB_{s}^{x},\,\,\,\, t\geq0,\,\, x\in\mathbb{Z}^{d},\\
v_{t}(x)=v_{0}(x)+\intop_{0}^{t}v_{s}Q(x)ds+\intop_{0}^{t}\sqrt{\widetilde{\gamma}u_{s}(x)v_{s}(x)}dW_{s}^{x},\,\,\,\, t\geq0,\,\, x\in\mathbb{Z}^{d},
\end{array}\end{cases}\label{eq:DP-eq}
\end{equation}
where $\{B_{s}^{x}\}_{x\in\mathbb{Z}^{d}},\,\{W_{s}^{x}\}_{x\in\mathbb{Z}^{d}}$
are collections of one-dimensional independent Brownian motions, and
$\tilde{\gamma}>0$.

One of the main questions which was introduced in \cite{DawsonPerkins98}
is the question of coexistence and non-coexistence of types in the
long run. In particular, it has been proved that there is a
clear dichotomy: coexistence is possible if the migration is transient, that is in dimensions $d\geq 3$, and is impossible if the migration is recurrent that is, if  $d\leq 2$.

 The above model is a particular case of so called  ``interacting mutually catalytic diffusions'' studied by many authors in different settings, see for example,
\cite{Mytnik98}, \cite{EJP114}, 
\cite{CoxKlenkePerkins00} among many others.


Cox, Dawson, Greven in \cite{CDG04} analyze the behavior of the Dawson-Perkins
system with very large but finite space of sites in comparison to
the corresponding model with infinite space of sites. This type of
question arises if, for example, one is interested in determining
what simulations of finite systems can say about the corresponding
infinite spatial system. In \cite{CDG04}, the authors consider a
sequence of finite subsets of $\mathbb{Z}^{d}$ increasing to the
whole $\mathbb{Z}^{d}$ and check the limiting behavior of mutually
catalytic models restricted to these sets, while time is also suitably
rescaled. It is called a ``finite system scheme''. This concept
appeared in \cite{CG1}, \cite{CGSh98a},  \cite{CGSh98b}. 

To formulate a result from \cite{CDG04} we need to introduce the
following construction. Fix $n\in\mathbb{N}$. Define $\Lambda_{n}=\left[-n,n\right]^{d}\cap\mathbb{Z}^{d}$.
Let $Q=\left(q(i,j)\right)_{i,j\in\mathbb{Z}^{d}}$ be the $Q$-matrix
of $\mathbb{Z}^{d}$-valued Markov chain. Define $Q^{n}=\left(q^{n}(i,j)\right)_{i,j\in\Lambda_{n}}$
as follows
\begin{equation}
q^{n}(0,y-x)=\sum_{z\in I(y-x)}q(0,z),\label{eq:transition_kernel}
\end{equation}
where
\[
I(x)=\left\{ y\in\mathbb{Z}^{d}\left|y=x\,\mod\,(\Lambda_{n})\right.\right\} .
\]
Consider a process that solves Dawson-Perkins equations \eqref{eq:DP-eq}
with state space being the torus $\Lambda_{n}$, i.e.
\begin{equation}
\begin{cases}
\begin{array}{c}
u_{t}^{n}(x)=u_{0}^{n}(x)+\intop_{0}^{t}u_{s}^{n}Q^{n}(x)ds+\intop_{0}^{t}\sqrt{\widetilde{\gamma}u_{s}^{n}(x)v_{s}^{n}(x)}dB_{s}^{x},\,\,\,\, t\geq0,\,\, x\in\Lambda_{n},\\
v_{t}^{n}(x)=v_{0}^{n}(x)+\intop_{0}^{t}v_{s}^{n}Q^{n}(x)ds+\intop_{0}^{t}\sqrt{\widetilde{\gamma}u_{s}^{n}(x)v_{s}^{n}(x)}dW_{s}^{x},\,\,\,\, t\geq0,\,\, x\in\Lambda_{n},
\end{array}\end{cases}\label{eq:DP-Lambda_n}
\end{equation}
where $\{B_{s}^{x}\}_{x\in\Lambda_{n}},\,\{W_{s}^{x}\}_{x\in\Lambda_{n}}$
are collections of one-dimensional independent Brownian motions, and
$\tilde{\gamma}>0$.

We denote such a process by $\left(U_{t}^{n},V_{t}^{n}\right)_{t\geq0}:=\left(\left(u_{t}^{n}(x)\right)_{x\in\Lambda_{n}},\left(v_{t}^{n}(x)\right)_{x\in\Lambda_{n}}\right)_{t\geq0}$.
Define time scaling depending on the system size
\[
\beta_{n}(t)=\left|\Lambda_{n}\right|t=(2n+1)^{d}t.
\]
Define:
\begin{equation}
\mathbf{U}_{t}^{n}=\sum_{x\in\Lambda_{n}}u_{t}^{n}(x),\,\,\,\,\,\,\mathbf{V}_{t}^{n}=\sum_{x\in\Lambda_{n}}v_{t}^{n}(x),\label{eq:sumU}
\end{equation}
and renormalized process
\begin{equation}
D_{n}((U_{\cdot}^{n},V_{\cdot}^{n}))=(D_{n}^{1},D_{n}^{2})=\frac{1}{\left|\Lambda_{n}\right|}\left(\sum_{x\in\Lambda_{n}}u_{\cdot}^{n}(x),\sum_{x\in\Lambda_{n}}v_{\cdot}^{n}(x)\right).\label{eq:D(U,V)}
\end{equation}
In what follows $\mathcal{L}(\cdot)$ denotes the law of random variable
or process.

It is shown in \cite{CDG04} that in dimensions $d\geq 3$  the sequence $D_{n}$-processes
with suitably rescaled time is tight and converges to a diffusion.
\begin{thm}
\label{thm:CDG-1}(Theorem 1(a) in \cite{CDG04}) Let $d\geq3$, and
let $Q$ be a generator of a simple random walk on $\mathbb{Z}^{d}$.
Assume that 
\[
u_{0}^{n}(x)=\theta_{1},v_{0}^{n}(x)=\theta_{2},\,\,\forall x\in\Lambda_{n}.
\]
Then 

\[
\mathcal{L}\left(D_{n}\left(U_{\beta_{n}(t)}^{n},V_{\beta_{n}(t)}^{n}\right)_{t\geq0}\right)\longrightarrow\mathcal{L}\left(\left(X_{t},Y_{t}\right)_{t\geq0}\right),\,\, as\, n\rightarrow\infty,
\]
where $\left(X_{t},Y_{t}\right)_{t\geq0}$ is the unique weak solution
for the following system of stochastic differential equations
\begin{equation}
\begin{cases}
dX_{t} & =\sqrt{\tilde{\gamma}X_{t}Y_{t}}dw^{1}(t),\,\,\, t\geq0,\\
dY_{t} & =\sqrt{\tilde{\gamma}X_{t}Y_{t}}dw^{2}(t),\,\,\, t\geq0
\end{cases}\label{eq:X_Y_law-1}
\end{equation}
with initial conditions $(X_{0},Y_{0})=\bar{\theta}=(\theta_1,\theta_2)$, where $w^{1},\, w^{2}$
are two independent standard Brownian motions.
\end{thm}



In this paper we consider the Dawson-Perkins mutually catalytic model
for {\it particle systems} and study its properties. As for other particle
models in the presence of interactions, they have been considered
earlier by many authors. A partial list of examples follows.

M. Birkner in \cite{Birkner2003} studies a system of particles performing
random walks on $\mathbb{Z}^{d}$ and performing branching; the rate of
branching of any particle on a site depends on the number of other
particles at the same site (this is the ``catalytic'' effect). Birkner
introduces a formal construction of such processes, via solutions
of certain stochastic equations, proves existence and uniqueness theorem
for these equations, and studies properties of the processes. Under
suitable assumptions he proves the existence of an equilibrium distribution
for shift-invariant initial conditions. He also studies survival and
extinction of the process in the long run. Note that the construction
of the process in \cite{Birkner2003} is motivated by the construction
of Ligget-Spitzer in \cite{LiggetSpitzer}.

Among many other works where branching particle systems with catalysts
were studied we can mention \cite{KestenSidoravicius03} and \cite{LiMa08}.
For example, Kesten and Sidoravicius in \cite{KestenSidoravicius03}
investigate the survival/extinction of two particle populations A
and B. Both populations perform an independent random walk. The B-particles
perform a branching random walk, but with a birth rate of new particles
which is proportional to the number of A-particles which coincide
with the appropriate B-particles. It is shown that for a choice of
parameters the system becomes (locally) extinct in all dimensions.

In \cite{LiMa08} catalytic discrete state branching processes with
immigration are defined as strong solutions of stochastic integral
equations. Z. Li and C.~Ma in \cite{LiMa08} prove limit theorems
for these processes.

In this paper we consider two interactive populations --- to be more precise we construct so-called mutually-catalytic branching model and study its long time behavior and   finite systems  scheme.

\subsection{Paper overview\label{sub:Thesis-overview}}

In the next two subsections we introduce our model and state main
results. In Section~\ref{sec:Model_MR} the process is formally constructed and main results are stated.
Sections~\ref{sec:ex_un}---\ref{chap:Finite-System-Scheme} are devoted to the proofs of our results. 

\medskip
\noindent\textbf{Acknowledgements.}
LM is supported in part by ISF grants No. 1704/18 and 1985/22.

\section{\label{sec:Model_MR}Our Model and Main Results}

\subsection{\label{sec:Model}Description of the Model}

Let us define the following interactive particle system. We consider
two populations on a countable set of sites $S\subset\mathbb{Z}^{d}$,
where particles in both populations move as independent Markov chains
on $S$ with rate of jumps $\kappa>0$, and symmetric transition jump probabilities
\begin{equation}
\label{eq:sym_p}
p_{x,y}=p_{y,x},\,\,\, x,y\in S.
\end{equation}
They also undergo branching events. In order to define our model formally
we are following the ideas of \cite{Birkner2003}.

Let $\left\{ \nu_{k}\right\} _{k\geq0}$ be the branching law. Suppose
that $Z$ is a random variable distributed according to $\nu$. We
assume that branching law is critical and has a finite variance: 
\begin{equation}
\mathbb{E}(Z)=\sum_{k\geq0}k\nu_{k}=1,\,\,\, Var(Z)=\sum_{k\geq0}(k-1)^{2}\nu_{k}=\sigma^{2}<\infty.\label{eq:variance-assumptions}
\end{equation}
The pair of processes $(\xi,\eta)$ describes the time evolution of
the following ``particle'' model. Between branching events in $\xi$
and $\eta$ populations move as independent Markov chains on $S$
with rate of jumps $\kappa$ and transition probabilities $p_{xy}$,
$x,y\in S$. Fix some $\gamma>0$. The ``infinitesimal'' rate of
a branching event for a particle from population $\xi$ at site $x$
at time $t$ equals to $\gamma\eta_{t}(x)$; similarly the ``infinitesimal''
rate of a branching event for a particle from population $\eta$ at
site $x$ at time $t$ equals to $\gamma\xi_{t}(x)$. When a ``branching
event'' occurs, a particle dies and is replaced by a random number
of offspring, distributed according to the law $\left\{ \nu_{k}\right\} _{k\geq0}$,
independently from the history of the process. To define a process
formally, as a solution to a system of equations we need more notations. Note that construction of the process follows the steps in~\cite{Birkner2003}.

The Markov chain is defined in the following way: Let $(W_{t},P)$
be a continuous time $S$-valued Markov chain  with rate of jumps $\kappa>0$, and symmetric transition jump probabilities
$
p_{x,y}=p_{y,x},\,x,y\in S.
$ Set $p_{t}(x,y)=P(W_{t}=y|W_{0}=x)$
as transitions probabilities. Let $Q=(q_{x,y})$ denote the associated
$Q$-matrix; that is $q_{x,y}=\kappa p_{x,y}$ is the jump rate from $x$ to $y$
(for $x\neq y$) and $q_{x,x}=-\sum_{y\neq x}q_{x,y}=-\kappa>-\infty$. Clearly, by our assumptions on transition jump probabilities,
 the $Q$-matrix is symmetric ($q_{x,y}=q_{y,x}$).
Define the Green function for every $x,y\in S$: 
\begin{equation}
g_{t}(x,y)=\intop_{0}^{t}p_{s}(x,y)ds.\label{eq:green-f-def}
\end{equation}
Note that if our motion process is a symmetric random walk on $S$,
hence, with certain abuse of notation, $g_{t}(x,y)=g_{t}(y,x)=g_{t}(x-y)$,
in particular $g_{t}(x,x)=g_{t}(0)$.

Let $P_{t}f(x)=\sum_{y}p_{t}(x,y)f(y)$ be the semigroup associated
with 
the Markov chain $W$ and $Qf(x)=\sum_{y}q_{x,y}f(y)$ is its generator.
\begin{rem}
If $W$ is a symmetric random walk, then
clearly $g_{\infty}(0)<\infty$ means that $W$ is transient, and
$g_{\infty}(0)=\infty$ implies that $W$ is recurrent. 
\end{rem}
Let $\mathcal{F}=\left(\mathcal{F}_{t}\right)_{t\geq0}$ be a (right-continuous,
complete) natural filtration. In what follows, when we call a process
martingale, we mean that it is an $\mathcal{F}_{t}$-martingale. 

Let
\[
\{N_{x,y}^{RW_{\xi}}\}_{x,y\in S, x\not=y},\,\,\,\{N_{x,y}^{RW_{\eta}}\}_{x,y\in S, x\not=y},\,\,\,\{N_{x,k}^{br_{\xi}}\}_{x\in S,k\in\mathbb{Z}_{+}},\,\,\,\{N_{x,k}^{br_{\eta}}\}_{x\in S,k\in\mathbb{Z}_{+}}
\]
denote independent Poisson point processes on $\mathbb{R}_{+}\times\mathbb{R}_{+}$
. We assume that, for any $x,y\in S,\,\, x\neq y$ both Poisson point
processes $N_{x,y}^{RW_{\xi}}$ and $N_{x,y}^{RW_{\eta}}$ have intensity
measure $\kappa p_{x,y}ds\otimes du$. Similarly, we assume that,
for any $x\in S,\,\, k\in\mathbb{Z}_{+}$ both Poisson point processes
$N_{x,k}^{br_{\xi}}$ and $N_{x,k}^{br_{\eta}}$ have intensity measure
$\nu_{k}ds\otimes du$. We assume that the above Poisson processes
are $\mathcal{F}$-adapted in the ``time'' $s$-coordinate.

Now we are going to define the pair of processes $(\xi_{t},\eta_{t})_{t\geq0}$ where
 $(\xi_{t},\eta_{t})\in\mathbb{N}_{0}^{S}\times\mathbb{N}_{0}^{S}$, and $\mathbb{N}_{0}$  denotes the set of non-negative integers.

For any $x\in S$, $\xi_{t}(x)$ counts the number of particles from
the first population at site $x$ at time $t$. Similarly, for any
$x\in S$, $\eta_{t}(x)$  counts the number of particles from the second
population at site $x$ at time $t$.

Now we are ready to describe $(\xi_{t},\eta_{t})_{t\geq0}$ formally
as a solution of the following system of equations:

\begin{align}
\xi_{t}(x)  = & \xi_{0}(x)+\sum_{y\neq x}\left\{ \intop_{0}^{t}\intop_{\mathbb{R}_{+}}1_{\{\xi_{s-}(y)\geq u\}}N_{y,x}^{RW_{\xi}}(dsdu)-\intop_{0}^{t}\intop_{\mathbb{R}_{+}}1_{\{\xi_{s-}(x)\geq u\}}N_{x,y}^{RW_{\xi}}(dsdu)\right\} \nonumber \\
 &   +\sum_{k\geq0}\intop_{0}^{t}\intop_{\mathbb{R}_{+}}(k-1)1_{\{\gamma\eta_{s-}(x)\xi_{s-}(x)\geq u\}}N_{x,k}^{br_{\xi}}(dsdu)\:,\: t\geq0,\, x\in S,\label{eq:main_equation}\\
\eta_{t}(x)  = & \eta_{0}(x)+\sum_{y\neq x}\left\{ \intop_{0}^{t}\intop_{\mathbb{R}_{+}}1_{\{\eta_{s-}(y)\geq u\}}N_{y,x}^{RW_{\eta}}(dsdu)-\intop_{0}^{t}\intop_{\mathbb{R}_{+}}1_{\{\eta_{s-}(x)\geq u\}}N_{x,y}^{RW_{\eta}}(dsdu)\right\} \nonumber \\
 &   +\sum_{k\geq0}\intop_{0}^{t}\intop_{\mathbb{R}_{+}}(k-1)1_{\{\gamma\xi_{s-}(x)\eta_{s-}(x)\geq u\}}N_{x,k}^{br_{\eta}}(dsdu)\:,\: t\geq0,\, x\in S.\nonumber 
\end{align}

Why do these equations actually describe our processes? The first
sum on the right-hand side of equations for $\xi$ and $\eta$ describes
the random walks of particles, and the second sum describes their
branching. The first integrals in the first sums describe all particles
jumping to site $x$ from different sites $y\neq x$.  The second
integrals in the first sum describe particles that leave site $x$.
The last integral describes the death of a particle at site $x$ and
the birth of its $k$ offsprings, so after that event the number of particles at the site has changed by 
$k-1$. The branching events at site $x$
happen with the infinitesimal rate proportional to the product of
the number of particles of both populations at site $x$.
\begin{defn}
The process $(\xi_{t},\eta_{t})$ solving \eqref{eq:main_equation}
is called a \emph{mutually catalytic branching process} with initial
conditions $(\xi_{0},\eta_{0})$.\end{defn}

\subsection{Main Results}

We start with stating the result on the existence and uniqueness of
the solution for the system of equation \eqref{eq:main_equation}.
This implies that the process we described in the introduction does
exist and is defined uniquely via the solution to \eqref{eq:main_equation}.
In the next theorem, we formulate the result for finite initial conditions,
i.e. each population has a finite number of particles at initial time
($t=0$). First, we introduce another piece of notation.
For $m\in\mathbb{N}$, define $L^{m}$-norm of $\varphi\in\mathbb{Z}^{S}$:
\begin{equation}
\left\Vert \varphi\right\Vert _{m}:=\left(\sum_{i\in S}|\varphi(i)|^{m}\right)^{1/m}.\label{eq:def-norma_1}
\end{equation}
Similarly, for any  $(\varphi,\psi)\in\mathbb{Z}^{S}\times\mathbb{Z}^{S}$,  $(\varphi,\psi,\tilde\varphi,\tilde\psi)\in\left(\mathbb{Z}^{S}\right)^4$, with some abuse of
notation, we define
\begin{equation}
\left\Vert \left(\varphi,\psi\right)\right\Vert _{m}:=\left(\sum_{i\in S}\left(|\varphi(i)|^{m}+|\psi(i)|^{m}\right)\right)^{1/m}.\label{eq:def-norma}
\end{equation}
\[
\left\Vert \left(\varphi,\psi,\tilde\varphi,\tilde\psi\right)\right\Vert _{m}:=\left(\sum_{i\in S}\left(|\varphi(i)|^{m}+|\psi(i)|^{m}+|\tilde\varphi(i)|^{m}+|\tilde\psi(i)|^{m}\right)\right)^{1/m}.
\]
In addition let us  define the space of functions $E_{fin}$:
\[
E_{fin}=\left\{ f:S\to\mathbb{N}_{0}|\left\Vert f\right\Vert _{1}<\infty\right\}.
\]
We equip $E_{fin}$ with the metric: $d_{E_{fin}}(f,g)= \left\Vert f-g \right\Vert _{1}$ for any $f,g\in E_{fin}$.



\begin{thm}
\label{lem:Lemma-1}Let $S\subset\mathbb{Z}^{d}$. a) For any initial
conditions $\mbox{\ensuremath{(\xi_{0},\eta_{0})\in E_{fin}\times E_{fin}}}$
there is a unique strong solution $(\xi_{t},\eta_{t})_{t\geq0}$ to
\eqref{eq:main_equation}, taking values in $E_{fin}\times E_{fin}$.

b) The solution $\left\{ (\xi_{t},\eta_{t}),\,\, t\geq0\right\} $
to \eqref{eq:main_equation} is a Markov process.
\end{thm}
It is possible to generalize the result to some infinite mass initial conditions
case but since this is not the goal of this paper it will be done elsewhere.

Let $(\xi,\eta)$ be the process constructed in Theorem \ref{lem:Lemma-1}
with finite initial conditions. Denote 
\[
\boldsymbol{\xi}_t=
\sum_{i\in S}\xi_t(i) ,\,\,\, \boldsymbol{\eta}_{t}=
\sum_{i\in S}\eta_t(i) ,\,\,\, t\geq0.
\]
That is, $\boldsymbol{\xi}$ is the total mass process of $\xi$, and $\boldsymbol{\eta}$
is the total mass process of $\eta$. Clearly by construction, $\boldsymbol{\xi}$
and $\boldsymbol{\eta}$ are non-negative local martingales and hence by the martingale
convergence theorem there exist an a.s. limits
\[
\boldsymbol{\xi}_{\infty}=\lim_{t\rightarrow\infty}\boldsymbol{\xi}_t,\,\,\,\boldsymbol{\eta}_{\infty}=\lim_{t\rightarrow\infty}\boldsymbol{\eta}_{t}.
\]
Now we are ready to give a definition of coexistence or non-coexistence.
\begin{defn}
\label{Coexistence-is-possible. Finite-1}Let $(\xi,\eta)$ be a unique
strong solution to \eqref{eq:main_equation} with $(\xi_{0},\eta_{0})\in E_{fin}\times E_{fin}$.
We say that \emph{coexistence is possible} for $(\xi,\eta)$ if $\mbox{\ensuremath{\mathbb{P}(\boldsymbol{\xi}_{\infty}\boldsymbol{\eta}_{\infty}>0)>0}}$.
We say that \emph{coexistence is impossible} for $(\xi,\eta)$ if
$\mathbb{P}(\boldsymbol{\xi}_{\infty}\boldsymbol{\eta}_{\infty}>0)=0$.
\end{defn}

\begin{description}
 \item [\hspace*{-1cm} {Convention}]We say that the motion process for the \emph{mutually
catalytic branching process} on $S=\mathbb{Z}^{d}$, is the nearest
neighbor random walk if
\[
p_{x,y}=\frac{1}{2d}\,\,\,\mathrm{for}\,\, y=x\pm e_{i},
\]
for $e_{i}$ a unit vector in an axis direction, $i=1,...,d$.
\end{description}
We will prove that in the finite initial conditions case, with motion
process being the nearest neighbor random walk, the coexistence is
possible if and only if the random walk is transient. Recall that
the nearest neighbor random walk is recurrent in dimensions $d=1,2$,
and it is transient in dimensions $d\geq3$. Then we have the following
theorem. 
\begin{thm}
\label{thm:Coexistance-finite-1} Let $S=\mathbb{Z}^{d}$ and assume
that the motion process is the nearest neighbor random walk. Let $(\xi_{0},\eta_{0})\in E_{fin}\times E_{fin}$ with $\boldsymbol{\xi}_{0}\boldsymbol{\eta}_{0}>0$. 

(a) If $d\geq3$, then coexistence of types is possible.

(b) If $d\leq2$, then coexistence of types is impossible.
\end{thm}
The proof is simple and based on the following observation: if there
is a finite number of particles and the motion is recurrent --- the
particles will meet an infinite number of times, and eventually one
of the populations dies out, due to the criticality of the branching
mechanism. On the other hand, if the motion is transient --- there
exists a finite time such that after this time the particles of different
populations never meet, and hence there is a positive probability
of survival of both populations.

Finally we are interested in a finite system scheme. We construct
a system of renormalized processes started from an exhausting sequence
of finite subsets of $\mathbb{Z}^{d}$, $\Lambda_{n}\subset\mathbb{Z}^{d}$.
The duality techniques will be used to investigate its limiting behavior. 

Define 
\begin{eqnarray*}
\Lambda_{n} & = & \left\{ x\in\mathbb{Z}^{d}\,|\,\forall i=1,...,d,\left|x_{i}\right|\leq n\right\} \subseteq\mathbb{Z}^{d},\\
\left|\Lambda_{n}\right| & = & (2n+1)^{d}.
\end{eqnarray*}

\begin{description}
\item [\hspace*{-1cm}{Convention}] Let $S=\Lambda_{n}$. We say that the motion process
is the nearest neighbor random walk on $\Lambda_{n}$ if its transition
jump probabilities are given by
\[
p_{x,y}^{n}=p_{0,y-x}^{n}=\begin{cases}
\frac{1}{2d}, & \mathrm{if}\,\left|x-y\right|=1,\\
0, & \mathrm{otherwise,}
\end{cases}
\]
where ``$y-x$'' is the difference on the torus $\Lambda_{n}$.
\end{description}
Fix $\bar{\theta}=\left(\theta_{1},\theta_{2}\right)$ with $\theta_{1},\theta_{2} \in\mathbb{N}_{0}$.
Assume $\bar{\boldsymbol{\theta}}=\left(\boldsymbol{\theta}_{1},\boldsymbol{\theta}_{2}\right)$,
where $\mbox{\ensuremath{\boldsymbol{\theta}_{i}=(\theta_{i},\theta_{i},...)\in\mathbb{N}_{0}^{\Lambda_{n}}}}$,
$i=1,2$. Let $(\xi_{t},\eta_{t})_{t\geq0}$ be the mutually catalytic
branching process with initial conditions $(\xi_{0},\eta_{0})=\bar{\boldsymbol{\theta}}$,
and site space $S=\Lambda_{n}$, and motion process being the nearest
neighbor walk on $\Lambda_{n}$.

Set

\begin{equation}
\boldsymbol{\xi}_{t}^{n}=\sum_{j\in\Lambda_{n}}\xi_{j}(t),\,\,\,\,\,\,\boldsymbol{\eta}_{t}^{n}=\sum_{j\in\Lambda_{n}}\eta_{j}(t).\label{eq:xi_N_def-1}
\end{equation}
We define the following time change:
\[
\beta_{n}(t)=\left|\Lambda_{n}\right|t,\,\,\, t\geq0.
\]
Our goal is to identify the limiting distribution of
\[
\frac{1}{\left|\Lambda_{n}\right|}\left(\boldsymbol{\xi}_{\beta_{n}(t)}^{n},\boldsymbol{\eta}_{\beta_{n}(t)}^{n}\right),
\]
as $n\rightarrow\infty$, for all $t\geq0$.
\begin{thm}
\label{thm:MainResult-1}Let $d\geq3$, and assume that 
\begin{equation}
\label{eq:gamsigma_bound}
\gamma\sigma^{2}<\frac{1}{\sqrt{3^{5}}(\frac12 g_{\infty}(0)+1)}, 
\end{equation}
 and $\sum_{k}k^{3}\nu_{k}<\infty.$
Then for any $T\in (0,1]$, we have 
\[
\mathcal{L}\left(\frac{1}{\left|\Lambda_{n}\right|}\left(\boldsymbol{\xi}_{\beta_{n}(T)}^{n},\boldsymbol{\eta}_{\beta_{n}(T)}^{n}\right)\right)\longrightarrow\mathcal{L}\left(X_{T},Y_{T}\right),\,\, as\,\, n\rightarrow\infty,
\]
where $\left(X_{t},Y_{t}\right)_{t\geq0}$ is a solution of the following
system of stochastic differential equations
\begin{equation}
\begin{cases}
dX_{t}=\sqrt{\gamma\sigma^{2}X_{t}Y_{t}}dw^{1}(t), & t\geq0,\\
dY_{t}=\sqrt{\gamma\sigma^{2}X_{t}Y_{t}}dw^{2}(t), & t\geq0,
\end{cases}\label{eq:XY-1}
\end{equation}
with initial conditions $(X_{0},Y_{0})=\bar{\theta}$, where $w^{1},\, w^{2}$
are two independent standard Brownian motions.
\end{thm}

\begin{rem}
The above theorem gives convergence of one-dimensional distributions of the rescaled processes $(\boldsymbol{\xi}^{n}, \boldsymbol{\eta}^{n})$ to the one-dimensional distributions of the solution of~\eqref{eq:XY-1} starting at  initial conditions 
$\bar{\theta}\in \mathbb{N}_{0}^2$. It seems possible to treat a more general class of initial conditions, for example,  i.i.d. configurations on 
$\Lambda_{n}$ with
mean vector $\bar{\theta}=(\theta_1, \theta_2)\in \mathbb{R}_+^2$.  However this will make the argument more technically involved, thus we decided to treat this case elsewhere. 
\end{rem}
\begin{rem}
The condition $\gamma\sigma^{2}<\frac{1}{\sqrt{\cdot3^{5}}(\frac12 g_{\infty}(0)+1)}$ arises from the method of proof, which requires boundedness of the fourth moment of the dual processes (see Lemma~\ref{cor:E(u^4)<infty}, where this condition is applied). We conjecture, however, that the result holds without the additional constraint on~$\gamma\sigma^2$, as is the case in the finite scheme for Dawson-Perkins processes.
\end{rem}
The above result is similar, although a bit weaker, to the result
in Theorem~1 in \cite{CDG04}, where a finite scheme for the system
of continuous stochastic differential equations (SDE's) is studied.
The proof of Theorem \ref{thm:MainResult-1} is based on the duality
principle for our particle system and the result for stochastic differential
equations in \cite{CDG04}.  In fact, let us mention that the self-duality property for our mutually catalytic branching particle model (the property which is well  known for processes solving equations of type~\eqref{eq:DP-Lambda_n}) does not hold. Thus, we use  the so called approximating duality technique to prove Theorem \ref{thm:MainResult-1}. The approximating duality technique was used in the past to resolve a number of weak uniqueness problems (see e.g.~\cite{bib:myt96}, \cite{bib:myt98a}).  
We believe that using approximate duality to prove limit theorems is novel and this technique is of independent interest. 

Let us note that it would be very interesting to extend the above results.  First of all it would be nice to address the question of coexistence/non-coexistence for more general motions and infinite mass initial conditions. As for extending results in Theorem \ref{thm:MainResult-1}, we would be interested to check what happens in the case of large $\gamma$, to prove  `functional convergence'' result as in \cite{CDG04}, and investigate the system's behavior for the case of recurrent  motion, that is,  in the dimensions $d=1,2$. We plan to address these problems in the future.

\section{Existence and Uniqueness.  Proof of Theorem~\ref{lem:Lemma-1}}
\label{sec:ex_un}


This section 
is devoted to the proof of Theorem \ref{lem:Lemma-1}.  Note that our proofs follow closely the argument of  Birkner \cite{Birkner2003} with suitable adaptation to the two types case.



For any metric space $D$ with metric $d$, let $\mathrm{Lip}(D)$
denote a set of Lipschitz functions on $D.$ We say that $f:D\rightarrow\mathbb{R}$
is in $\mathrm{Lip}(D)$ if and only if there exists a positive constant
$C\in\mathbb{R}_{+}$ such that for any $\varphi,\psi\in D$, $\left|f(\varphi)-f(\psi)\right|\leq Cd(\varphi,\psi)$.
Let $L^{(2)}$ denote the following operator: for a measurable function $f: E_{fin}\times E_{fin} \to \mathbb{R}$,  let 
\begin{eqnarray*}
L^{(2)}f(\varphi,\psi) & = & \kappa\sum_{x,y\in S}\varphi(x)p_{xy}\left(f(\varphi^{x\rightarrow y},\psi)-f(\varphi,\psi)\right)\\
 &  & +\kappa\sum_{x,y\in S}\psi(x)p_{xy}\left(f(\varphi,\psi^{x\rightarrow y})-f(\varphi,\psi)\right)\\
 &  & +\sum_{x\in S}\gamma\varphi(x)\psi(x)\sum_{k\geq0}\nu_{k}\left(f(\varphi+(k-1)\delta_{x},\psi)-f(\varphi,\psi)\right)\\
 &  & +\sum_{x\in S}\gamma\varphi(x)\psi(x)\sum_{k\geq0}\nu_{k}\left(f(\varphi,\psi+(k-1)\delta_{x})-f(\varphi,\psi)\right),
\end{eqnarray*}
where $\varphi^{x\rightarrow y}=\varphi+\delta_{y}-\delta_{x}$, i.e. $\varphi^{x\rightarrow y}(x)=\varphi(x)-1$,
$\varphi^{x\rightarrow y}(y)=\varphi(y)+1$ and $\varphi^{x\rightarrow y}(z)=\varphi(z)$
for all $z\in S$ and $z\neq x,y$.

Theorem \ref{lem:Lemma-1} follows immediately from the next lemma.
\begin{lem}
\label{lem:Lemma}a) For any initial conditions $(\xi_{0},\eta_{0})\in E_{fin}\times E_{fin}$
there is a unique strong solution $(\xi_{t},\eta_{t})_{t\geq0}$ to
\eqref{eq:main_equation}, taking values in $E_{fin}\times E_{fin}$.

b) The solution $\left\{ (\xi_{t},\eta_{t}),\,\, t\geq0\right\} $
to \eqref{eq:main_equation} is a Markov process.

c) Let $m\in\mathbb{N}$. If $\sum_{k}k^{m}\nu_{k}<\infty$, there
exists a constant $C_{m}$ such that
\begin{equation}
\mathbb{E}\left[\left\Vert \left(\xi_{t},\eta_{t}\right)\right\Vert _{m}^{m}\right]\leq\exp(C_{m}t)\left\Vert \left(\xi_{0},\eta_{0}\right)\right\Vert _{m}^{m}.\label{eq:thm1ineq}
\end{equation}

d) 
For $f\in\mathrm{Lip}(E_{fin}\times E_{fin}),\,\,\,(\xi_{0},\eta_{0})\in E_{fin}\times E_{fin}$,
\begin{equation}
M^{f}(t):=f(\xi_{t},\eta_{t})-f(\xi_{0},\eta_{0})-\intop_{0}^{t}L^{(2)}f(\xi_{s},\eta_{s})ds,\label{eq:Mf}
\end{equation}
is a martingale. Moreover, if $f\in\mathrm{Lip}(\mathbb{R}_{+}\times E_{fin}\times E_{fin})$
and there is constant $C^{*}$ such that 
\begin{equation}
\left|\frac{\partial}{\partial s}f(s,\varphi,\psi)\right|\leq C^{*}\left\Vert (\varphi,\psi)\right\Vert _{2}^{2}\label{eq:partial_f_condition}
\end{equation}
 for any $\varphi,\psi\in E_{fin}$, then
\begin{equation}
N^{f}(t):=f(t,\xi_{t},\eta_{t})-f(0,\xi_{0},\eta_{0})-\intop_{0}^{t}\left[L^{(2)}f(s,\xi_{s},\eta_{s})+\frac{\partial}{\partial s}f(s,\xi_{s},\eta_{s})\right]ds,\label{eq:N^f-def}
\end{equation}
is also a martingale.\end{lem}
\begin{proof}
Note that in our proof we follow ideas of the proof of Lemma~1 in Birkner~\cite{Birkner2003}. 

\textbf{a)} $\{N_{x,y}^{RW_{\xi}}\}_{x,y\in S},\,\,\,\{N_{x,y}^{RW_{\eta}}\}_{x,y\in S},\,\,\,\{N_{x,k}^{br_{\xi}}\}_{x\in S,k\in\mathbb{Z}_{+}},\,\,\,\{N_{x,k}^{br_{\eta}}\}_{x\in S,k\in\mathbb{Z}_{+}}$
are a collection of independent Poisson point processes. Therefore
with probability 1 there is no more than one jump simultaneously.
Then we can define a stopping time $T_{1}$ --- the first time a jump
happens. The stopping time $T_1$ is stricly positive almost surely, since $(\xi_{0},\eta_{0})\in E_{fin}\times E_{fin}$, and therefore the indicators $1_{\{\xi_{s-}(y)\geq u\}}, 1_{\{\eta_{s-}(y)\geq u\}}, 1_{\{\gamma\xi_{s-}(x)\eta_{s-}(x)\geq u\}}$ make the rate of the first jump finite. 
Then $(\xi_{t},\eta_{t})=(\xi_{0},\eta_{0})$ for $t\in[0,T_{1})$.
In the same way we define a sequence of stopping times: $0=T_{0}<T_{1}<T_{2}<\cdots<\infty$, and again the rate of jumps is finite since for any $i\geq 1$,  $(\xi_{T_{i}},\eta_{T_{i}})\in E_{fin}\times E_{fin}$ almost surely, and this makes the rate of the $(i+1)$-th jump almost surely finite by the same reasoning as for the first jump. 
Clearly, by construction, the process $(\xi_{t},\eta_{t})$
is constant on the intervals $[T_{i},T_{i+1})$. In order to show
that this construction defines the process properly (that is, it does
not explode in finite time), it is enough to show that $\lim_{n\rightarrow\infty}T_{n}=\infty$,
almost surely. To this end define $M_{n}=\sum_{x\in S}\left(\xi_{T_{n}}(x)+\eta_{T_{n}}(x)\right)$.
$M_{n}$ denotes the total number of particles in both populations
at time $T_{n}$. Since the branching mechanism is critical (see \eqref{eq:variance-assumptions}),
it is easy to see that $\left\{ M_{n}\right\} _{n\geq0}$ is a non-negative
martingale. 

Indeed, suppose that $T_{n}$ is the stopping time originated from  a ``random
walk'' jump, that is, from a jump of one of the processes
$$R_{t}^{\xi,x}= \sum_{y\neq x}\left\{ \intop_{0}^{t}\intop_{\mathbb{R}_{+}}1_{\{\xi_{s-}(y)\geq u\}}N_{y,x}^{RW_{\xi}}(dsdu)-\intop_{0}^{t}\intop_{\mathbb{R}_{+}}1_{\{\xi_{s-}(x)\geq u\}}N_{x,y}^{RW_{\xi}}(dsdu)\right\},  t\geq 0,  x\in S,$$
or 
$$R_{t}^{\eta,x}=
\sum_{y\neq x}\left\{ \intop_{0}^{t}\intop_{\mathbb{R}_{+}}1_{\{\eta_{s-}(y)\geq u\}}N_{y,x}^{RW_{\eta}}(dsdu)-\intop_{0}^{t}\intop_{\mathbb{R}_{+}}1_{\{\eta_{s-}(x)\geq u\}}N_{x,y}^{RW_{\eta}}(dsdu)\right\}, t\geq 0,  x\in S.
$$
 In that case the total
number of particles does not change (this can also be readily seen
from our equation \eqref{eq:main_equation} since $\sum_{x\in S} R_{t}^{\xi,x}=\sum_{x\in S} R_{t}^{\eta,x}=0$) and thus we have $M_{n}=M_{n-1}$.
Alternatively $T_{n}$ can be originated from the ``branching'',
that is from the jump of one of the processes $B_{t}^{\xi}=  \sum_{x\in S}\sum_{k\geq0}\intop_{0}^{t}\intop_{\mathbb{R}_{+}}(k-1)1_{\{\gamma\eta_{s-}(x)\xi_{s-}(x)\geq u\}}N_{x,k}^{br_{\xi}}(dsdu)$ 
or $B_{t}^{\eta}= \sum_{x\in S}\sum_{k\geq0}\intop_{0}^{t}\intop_{\mathbb{R}_{+}}(k-1)1_{\{\gamma\eta_{s-}(x)\xi_{s-}(x)\geq u\}}N_{x,k}^{br_{\eta}}(dsdu)$. In this case one can
easily get that 
\[
\mathbb{E}\left(M_{n}|\, M_{0},...,M_{n-1}\right)=M_{n-1}+\mathbb{E}\left(Z-1\right)=M_{n-1},
\]
where $Z$ is distributed according to the branching law $\nu$.

Therefore by the well known martingale convergence theorem $\sup_{n\geq1}M_{n}<\infty$
almost surely (Theorem 1.6.4 in \cite{IkedaWatanabe}). This implies
that $\sup_{n}T_{n}=\infty$, almost surely.

Now let us turn to the proof of uniqueness. Let $(\tilde{\xi},\tilde{\eta})_{t}$
be another solution to \eqref{eq:main_equation} starting from the
same initial conditions $\tilde{\xi}_{0}=\xi^{0},\,\,\tilde{\eta}_{0}=\eta^{0}$.
We see from \eqref{eq:main_equation} that $\tilde{\xi}_{t}(x)=\xi_{t}(x),\,\,\tilde{\eta}_{t}(x)=\eta_{t}(x)$
for all $x\in S$ and $t\in[0,T_{1})$, and also that $\tilde{\xi}_{T_{1}}=\xi_{T_{1}},\,\,\tilde{\eta}_{T_{1}}=\eta_{T_{1}}$.
Then, by induction, $(\xi,\eta)$ and $(\tilde{\xi},\tilde{\eta})$
agree on $[T_{n},T_{n+1})$ for all $n\in\mathbb{N}$.

\textbf{b)} Poisson processes have independent and stationary increments.
Therefore, by construction described in \textbf{(a)}, we can immediately
see that the distribution of $(\xi_{t+h},\eta_{t+h})$, given $\mathcal{F}_{t}$,
depends only on $(\xi_{t},\eta_{t})$, and hence the process $(\xi_{t},\eta_{t})_{t\geq0}$
is Markov.

\textbf{c)} Now we will show that $(\xi_{t},\eta_{t})_{t\geq0}$ satisfies
\eqref{eq:thm1ineq}. 
Define a  sequence of stopping times
\begin{equation}
\label{eq:tn_loc}
T_{n}:=\inf\left\{ t\geq0\,:\,\begin{array}{c}
\sum_{x\in S}\left(\xi_{t}(x)+\eta_{t}(x)\right)>n
\end{array}\right\},\;n\geq 1. 
\end{equation}
Choose arbitrarily $m\in\mathbb{N}$ such that 
\[
\sum_{k}k^{m}\nu_{k}<\infty.
\]
For $\varphi, \psi\in E_{fin}$, define 
\[
h_{m}(\varphi,\psi):=\left\Vert (\varphi,\psi)\right\Vert _{m}^{m}.
\]
Then, by an appropriate version of  the It\^o formula, we have that $\left\{ M_{t\wedge T_{n}}^{h_{m}}\right\}_{t\geq 0} $
is a martingale. Let $\varphi,\psi\in E_{fin}$. 
Apply the $L^{(2)}$ operator on $h_{m}(\varphi,\psi)$ to get 
\begin{eqnarray*}
L^{(2)}h_{m}(\varphi,\psi) & = & \kappa\sum_{x,y}\varphi(x)p_{xy}\left\{ \left((\varphi(y)+1)^{m}-\varphi(y)^{m}\right)\right.\\
 &  & \left.+\left((\varphi(x)-1)^{m}-\varphi(x)^{m}\right)\right\} \\
 &  & +\kappa\sum_{x,y}\psi(x)p_{xy}\left\{\left((\psi(y)+1)^{m}-\psi(y)^{m}\right)\right.\\
 &  & \left.+\left((\psi(x)-1)^{m}-\psi(x)^{m}\right)\right\} \\
 &  & +\sum_{x}\gamma\varphi(x)\psi(x)\sum_{k\geq 0}\nu_{k}\left\{ (\varphi(x)+k-1)^m-\varphi(x)^{m}\right\} \\
 &  & +\sum_{x}\gamma\varphi(x)\psi(x)\sum_{k\geq 0}\nu_{k}\left\{ (\psi(x)+k-1)^m-\psi(x)^{m}\right\} \\
 & = & \kappa\sum_{x}\varphi(x)\sum_{j=1}^{m}\binom{m}{j}(-1)^{j}\varphi(x)^{m-j}\\
 &  & +\kappa\sum_{x,y}\varphi(x)p_{xy}\sum_{j=1}^{m}\binom{m}{j}\varphi(y)^{m-j}\\
 &  & +\kappa\sum_{x}\psi(x)\sum_{j=1}^{m}\binom{m}{j}(-1)^{j}\psi(x)^{m-j}\\
 &  & +\kappa\sum_{x,y}\psi(x)p_{xy}\sum_{j=1}^{m}\binom{m}{j}\psi(y)^{m-j}\\
 &  & +\sum_{x}\gamma\varphi(x)\psi(x)\sum_{k\geq 0}\nu_{k}\sum_{j=2}^{m}\binom{m}{j}\varphi(x)^{m-j}(k-1)^{j}\\
 &  & +\sum_{x}\gamma\varphi(x)\psi(x)\sum_{k\geq 0}\nu_{k}\sum_{j=2}^{m}\binom{m}{j}\psi(x)^{m-j}(k-1)^{j}.
\end{eqnarray*}
 where in the last equality we used the binomial expansion, the fact that
$\sum_{y}p_{xy}=1$ and our assumptions $\sum_{k\geq0}(k-1)\nu_{k}=0.$
Now we can estimate
\begin{eqnarray}
\nonumber
\left|L^{(2)}h_{m}(\varphi,\psi)\right| & \leq & \kappa\left(\sum_{j=1}^{m}\binom{m}{j}\right)\left(h_{m}(\varphi,\psi)+\sum_{x,y}p_{xy}\left[\varphi(x)^m + \varphi(y)^{m}+\psi(x)^m+ \psi(y)^{m}\right]\right)\\
\label{eq:L2_b1} 
&  & +\left(3\gamma\sum_{j=2}^{m}\binom{m}{j}\sum_{k}\nu_{k}(k-1)^{j}\right)h_{m}(\varphi,\psi)
\end{eqnarray}
where we used the following simple inequalities: for $m\geq j\geq1$,
\[
\varphi(x)^{m-j}\leq\varphi(x)^{m-1}
\]
(recall that $\varphi(x)$ is a non-negative integer number),  and 
$$ ab^{m-1}\leq a^m+ b^m,\; \forall a,b\geq 0. $$
Since, by our assumptions,  $p_{xy}=p_{yx}$ for all $x,y\in S$, we have 
\begin{equation}
\label{eq:L2_b2} 
\sum_{x,y\in S}p_{xy}\left[\varphi(x)^m + \varphi(y)^{m}+\psi(x)^m+ \psi(y)^{m}\right]\leq 4h_{m}(\varphi,\psi). 
\end{equation}
Denote 
\[
c_{m}:=\sum_{j=1}^{m}\binom{m}{j}=2^{m}-1,\,\,\, c'_{m}:=3\gamma\sum_{j=2}^{m}\binom{m}{j}\sum_{k\geq 0}\nu_{k}(k-1)^{j}<\infty.
\]
Then by~\eqref{eq:L2_b1} and \eqref{eq:L2_b2}  we get 
\begin{align}
\left|L^{(2)}h_{m}(\varphi,\psi)\right|  \leq &
(5 \kappa c_{m}+c'_{m})h_{m}(\varphi,\psi)=:C_m h_{m}(\varphi,\psi).\label{eq:Lpsi}
\end{align}
Now recall that for any $n$, $\left\{M^{h_{m}}(t\wedge T_{n})\right\}_{t\geq0}$
is a martingale. Therefore, 
\begin{eqnarray*}
\mathbb{E}\left[h_{m}(\xi_{t\wedge T_{n}},\eta_{t\wedge T_{n}})\right] & = & h_{m}(\xi^{0},\eta^{0})+\mathbb{E}\left[\intop_{0}^{t\wedge T_{n}}L^{(2)}h_{m}(\xi_{s},\eta_{s})ds\right]\\
 & \leq & h_{m}(\xi^{0},\eta^{0})+C_{m}\intop_{0}^{t}\mathbb{E}\left[\mathbf{1}_{\left\{ s\leq T_{n}\right\} }h_{m}(\xi_{s},\eta_{s})\right]ds\\
 & \leq & h_{m}(\xi^{0},\eta^{0})+C_{m}\intop_{0}^{t}\mathbb{E}\left[h_{m}(\xi_{s\wedge T_{n}},\eta_{s\wedge T_{n}})\right]ds.
\end{eqnarray*}
Thus, from Gronwall's lemma we get that
\[
\mathbb{E}\left[h_{m}(\xi_{t\wedge T_{n}},\eta_{t\wedge T_{n}})\right]\leq\exp(C_{m}t)h_{m}(\xi^{0},\eta^{0}),
\]
 uniformly in $n$. It is easy to see from (a) that $T_n\to \infty$ as $n\to\infty$, a.s.  Thus,  inequality \eqref{eq:thm1ineq} follows from Fatou's
lemma by letting $n\rightarrow\infty$. 

\textbf{d)} Let $f\in\mathrm{Lip}(E_{fin}\times E_{fin})$. We wish to
show that $M^{f}$ is indeed a martingale. In order to do that, first
we show that for any such $f$ there is a constant $C=C(\kappa,p,\sigma,\nu,f)$
such that
\begin{equation}
\left|L^{(2)}f(\varphi,\psi)\right|\leq C\left\Vert (\varphi,\psi)\right\Vert _{2}^{2}\,\,\mathrm{for\, all\,\,}\varphi,\psi\in E_{fin}.\label{eq:operatorBound}
\end{equation}
We decompose $L^{(2)}f(\varphi,\psi)$ into two parts corresponding to
motion and branching mechanisms:
\begin{equation}
L^{(2)}f(\varphi,\psi)=L_{RW}f(\varphi,\psi)+L_{br}f(\varphi,\psi),\label{eq:L(2)=00003DL_rw+L_br}
\end{equation}
where 
\begin{align}
L_{RW}f(\varphi,\psi)  = & \kappa\sum_{x,y\in S}\varphi(x)p_{xy}\left(f(\varphi^{x\rightarrow y},\psi)-f(\varphi,\psi)\right)\label{eq:L_rw}\\
 &   +\kappa\sum_{x,y\in S}\psi(x)p_{xy}\left(f(\varphi,\psi^{x\rightarrow y})-f(\varphi,\psi)\right),\nonumber \\
L_{br}f(\varphi,\psi) = & \sum_{x\in S}\gamma\varphi(x)\psi(x)\sum_{k\geq0}\nu_{k}\left(f(\varphi+(k-1)\delta_{x},\psi)-f(\varphi,\psi)\right)\label{eq:L_br}\\
 &   +\sum_{x\in S}\gamma\varphi(x)\psi(x)\sum_{k\geq0}\nu_{k}\left(f(\varphi,\psi+(k-1)\delta_{x})-f(\varphi,\psi)\right).\nonumber 
\end{align}
Using the Lipshitz property of $f$ we obtain
\begin{eqnarray*}
\left|L_{RW}f(\varphi,\psi)\right| & = & \left|\sum_{x,y\in S}\varphi(x)p_{xy}\left(f(\varphi^{(x,y)},\psi)-f(\varphi,\psi)\right)\right.\\
 &  & \left.+\sum_{x,y\in S}\psi(x)p_{xy}\left(f(\varphi,\psi^{(x,y)})-f(\varphi,\psi)\right)\right|\\
 & \leq & C_{f}\sum_{x,y\in S}\varphi(x)p_{xy}\left\Vert (\varphi^{(x,y)},\psi)-(\varphi,\psi)\right\Vert _{1}\\
 &  & +C_{f}\sum_{x,y\in S}\psi(x)p_{xy}\left\Vert (\varphi,\psi^{(x,y)})-(\varphi,\psi)\right\Vert _{1}\\
 & \leq & 2C_{f}\sum_{x,y\in S}\varphi(x)p_{xy}+2C_{f}\sum_{x,y\in S}\psi(x)p_{xy}\\
 & = & 2C_{f}\left\Vert (\varphi,\psi)\right\Vert _{1}\leq 2C_{f}\left\Vert (\varphi,\psi)\right\Vert _{2}^{2}
\end{eqnarray*}
 where in the last inequality we used$\left\Vert (\varphi,\psi)\right\Vert _{1}\leq\left\Vert (\varphi,\psi)\right\Vert _{2}^{2}$,
which holds since the functions $\left\{ \varphi(x)\right\} _{x\in S}$
and $\left\{ \psi(x)\right\} _{x\in S}$ are integer valued. Turning
to $L_{br}f(\varphi,\psi)$ we get
\begin{eqnarray*}
\left|L_{br}f(\varphi,\psi)\right| & = & \left|\sum_{x\in S}\gamma\varphi(x)\psi(x)\sum_{k\geq 0}\nu_{k}\left(\left[f(\varphi+(k-1)\delta_{x},\psi)-f(\varphi,\psi)\right]\right.\right.\\
 &  & \left.\left.+\left[f(\varphi,\psi+(k-1)\delta_{x})-f(\varphi,\psi)\right]\right)\right|\\
 & \leq & C_{f}\sum_{x\in S}\gamma\varphi(x)\psi(x)\sum_{k\geq 0}\nu_{k}\left(2\left\Vert (k-1)\delta_{x}\right\Vert _{1}\right)\\
 & = & \left(2C_{f}\sum_{k\geq 0}\nu_{k}|k-1|\right)\sum_{x\in S}\gamma\varphi(x)\psi(x)\\
 & \leq & 2\gamma C_{f}\sum_{k\geq0}2\nu_{k}|k-1|\cdot\left\Vert (\varphi,\psi)\right\Vert _{2}^2,
\end{eqnarray*}
where in the last inequality we used the fact that $\varphi(x)\psi(x)\leq\varphi(x)^{2}+\psi(x)^{2}$.
Thus \eqref{eq:operatorBound} holds with 
\[
C:=C_{f}(2+2\gamma\sum_{k\geq0}\nu_{k}|k-1|).
\]
Consider a bounded $f\in\mathrm{Lip}(E_{fin}\times E_{fin})$, then $M^{f}$
is a local martingale, so we have for all $t,h\geq0$
\begin{equation}
\mathbb{E}\left[\left.M^{f}\left(\left(t+h\right)\wedge T_{n}\right)\right|\mathcal{F}_{t}\right]=M^{f}\left(\left(t+h\right)\wedge T_{n}\right).\label{eq:localMartProperty},
\end{equation}
where $T_n$ is defined in~\eqref{eq:tn_loc}. 
The right-hand side converges to $M^{f}(t)$ a.s., as $n\rightarrow\infty$.
Then
\[
\mathbb{E}\left[\left.\intop_{0}^{(t+h)\wedge T_{n}}L^{(2)}f(\xi_{s},\eta_{s})ds\right|\mathcal{F}_{t}\right]\rightarrow\mathbb{E}\left[\left.\intop_{0}^{t+h}L^{(2)}f(\xi_{s},\eta_{s})ds\right|\mathcal{F}_{t}\right],\,\,\,\mathrm{a.s},
\]
as $n\rightarrow\infty$. Here we used again the dominated convergence,
since by \eqref{eq:operatorBound} 
\[
\intop_{0}^{(t+h)\wedge T_{n}}L^{(2)}f(\xi_{s},\eta_{s})ds\leq C\intop_{0}^{t+h}\left\Vert (\xi_{s},\eta_{s})\right\Vert _{2}^{2}ds
\]
 and the expectation on the right-hand side of the above inequality
is bounded due to \eqref{eq:thm1ineq} and finite initial conditions, for which clearly $\left\Vert (\xi_0,\eta_0)\right\Vert _{2}^{2}<\infty$. 
Thus $M^{f}$ is
indeed a martingale in the case of bounded $f\in\mathrm{Lip}(E_{fin}\times E_{fin})$.

Next, consider $f\in\mathrm{Lip}(E_{fin}\times E_{fin})$ which is non-negative,
but not necessarily bounded. Define $f_{n}(\varphi,\psi):=f(\varphi,\psi)\wedge n$.
Note that $f_{n}$ is bounded and $f_{n}\in\mathrm{Lip}(E_{fin}\times E_{fin})$
with Lipshitz constant $C_{f_{n}}\leq C_{f}$. As $n\rightarrow\infty$,
we have
\[
\begin{array}{c}
M^{f_{n}}(t)\rightarrow M^{f}(t),\,\,\,\mathrm{a.s}.\\
\mathbb{E}\left[f_{n}(\xi_{t+h},\eta_{t+h})\,|\,\mathcal{F}_{t}\right]\rightarrow\mathbb{E}\left[f(\xi_{t+h},\eta_{t+h})\,|\,\mathcal{F}_{t}\right],\,\,\mathrm{a.s}.
\end{array}
\]
by monotone convergence. Observe that $\left|L^{(2)}f_{n}(\varphi,\eta)\right|\leq C\left\Vert (\varphi,\eta)\right\Vert _{2}^{2}$
uniformly in $n$. We thus obtain
\[
\mathbb{E}\left[\left.\intop_{0}^{t+h}L^{(2)}f_{n}(\xi_{s},\eta_{s})ds\right|\mathcal{F}_{t}\right]\rightarrow\mathbb{E}\left[\left.\intop_{0}^{t+h}L^{(2)}f(\xi_{s},\eta_{s})ds\right|\mathcal{F}_{t}\right]
\]
 a.s. as $n\rightarrow\infty$ by the dominated convergence theorem.
Therefore $M^{f}$ is a martingale for non-negative Lipschitz $f$.
For the general case we use the decomposition of $f\in\mathrm{Lip}(E_{fin}\times E_{fin})$
as $f=f^{+}-f^{-}$, where $f^{+}:=\max(f,0)$ and $f^{-}:=\max(-f,0)$.

The same proof holds for $N^{f}$ too, since $\frac{\partial}{\partial s}f$
~is bounded by \eqref{eq:partial_f_condition}.
\end{proof}

\section{ Proof of Theorem \ref{thm:Coexistance-finite-1}.\label{chap:Coexistence-and-non-coexistence}}


The aim of this section is to prove Theorem \ref{thm:Coexistance-finite-1}. 

Let $(\xi_{t},\eta_{t})$ be the mutually catalytic branching process
described in Theorem~\ref{thm:Coexistance-finite-1}, starting at
$(\xi_{0},\eta_{0})$ with 
with $\boldsymbol{\xi}_{0}+ \boldsymbol{\eta}_{0}<\infty.$
Recall that $\boldsymbol{\xi}_t,\, \boldsymbol{\eta}_t, t\geq 0,$ denote the total size of each
population at time $t$:
\[
\boldsymbol{\xi}_t=\sum_{x\in\mathbb{Z}^{d}}\xi_{t}(x)\,\,\,\mathrm{and}\,\,\, \boldsymbol{\eta}_t=\sum_{x\in\mathbb{Z}^{d}}\eta_{t}(x).
\]



\subsection{Proof of Theorem \ref{thm:Coexistance-finite-1}(a) $-$ Transient
case}



The proof is simple and we decided to avoid technical details. The observation is as follows:  since the motion of particles is transient  and the number of particles  in the original populations is finite,  there
exists almost surely a finite time $\hat T$ such that, if one suppresses the branching,  the initial particles of different
populations never meet after time $\hat T$. On the other hand due to the finiteness of number of particles, the total branching rate in the system is finite and thus there is a positive probability for the event that in the original particle system there is no branching event until time $\hat T$. On this event,  particles of different
populations never meet after time $\hat T$ and therefore there  is a positive probability of survival of both populations. 

\qed

\subsection{Proof of Theorem \ref{thm:Coexistance-finite-1}(b) $-$ Recurrent
case}

We would like to show that
\[
\boldsymbol{\xi}_{\infty}\boldsymbol{\eta}_{\infty}=0,\,\,\,\mathbb{P}-\mathrm{a.s.}
\]
First, recall why $\lim_{t\rightarrow\infty}\boldsymbol{\xi}_{t}\boldsymbol{\eta}_{t}=\boldsymbol{\xi}_{\infty}\boldsymbol{\eta}_{\infty}$
exists. 
By It\^o's  formula it is easy to see that  $\left\{ \boldsymbol{\xi}_{t}\boldsymbol{\eta}_{t}\right\} _{t\geq0}$
is a non-negative local  martingale, that is a non-negative supermartingale.  By the Martingale Convergence theorem
non-negative
supermartingales converge a.s. as time goes to infinity. Hence, 
\[
\lim_{t\rightarrow\infty}\boldsymbol{\xi}_{t}\boldsymbol{\eta}_{t}=\boldsymbol{\xi}_{\infty}\boldsymbol{\eta}_{\infty},\,\,\,\mathbb{P}-\mathrm{a.s.}
\]
Also note that $\left\{\boldsymbol{\xi}_{t}\boldsymbol{\eta}_{t}\right\} _{t\geq0}$ is
an integer-valued supermartingale. Therefore there exists a random time
$T_{0}$ such that 
\begin{equation}
\boldsymbol{\xi}_{t}\boldsymbol{\eta}_{t}=\boldsymbol{\xi}_{\infty}\boldsymbol{\eta}_{\infty}\,\,\mathrm{for\, all}\,\, t\geq T_{0}.\label{eq:Limit dist}
\end{equation}
Now assume that $\boldsymbol{\xi}_{\infty}\boldsymbol{\eta}_{\infty}>0$, that is $\boldsymbol{\xi}_{t}>0$
and $\boldsymbol{\eta}_{t}>0$ for $t\geq T_{0}$. Since the motion is recurrent,
there is probability one for a ``meeting''  of two populations
after time $T_{0}$ at some site. Moreover, on the event $\left\{\boldsymbol{\xi}_{\infty}\boldsymbol{\eta}_{\infty}>0\right\} $,
by recurrence, after time $T_{0}$, two populations spend an infinite
amount of time ``together''. Since the branching rate is at least
$\gamma>0$, when particles of two populations spend time ``together''
on the same site, we immediately get that eventually a branching event
will happen with probability one. However, this is a contradiction
with \eqref{eq:Limit dist}. Therefore, $\boldsymbol{\xi}_{t}=0$ or $\boldsymbol{\eta}_{t}=0$
for all $t\geq T_{0}$, that is one of the populations becomes extinct,
and coexistence is not possible.
\qed

\section{Moment computations for 
 $S=\Lambda_{n}$}

In this section we derive some useful moment estimates for $(\xi_{t},\eta_{t})$
solving \eqref{eq:main_equation} in the case of 
$S=\Lambda_{n}$ for arbitrary $n\geq1$  (recall that $\Lambda_{n}$ is the torus defined in Section
\ref{sub:Background-and-Motivation}). 
These estimates will be very essential for proving Theorem \ref{thm:MainResult-1} in Section~\ref{chap:Finite-System-Scheme}. 

To simplify the notations
we suppress dependence on ``$n$'' in the notations. 
Throughout the section the motion process for a mutually catalytic
process $(\xi_{t},\eta_{t})$ is the nearest neighbor random walk
on 
 $S=\Lambda_{n}$. The transition semigroup
(respectively to transition density $\left\{ p_{t}(\cdot,\cdot)_{t\geq0}\right\} $,
$Q$-matrix) of the motion process will be denoted by $\left\{ P_{t}\right\} _{t\geq0}$.
 Motion process will be
the nearest neighbor random walk on $S$. For the definition of conditional quadratic variation, see~Chapter~III of~\cite{bib:protter04}. For $\psi\in  \mathbb{Z}^S$ and $\varphi\in \mathbb{R}^S$ define the inner product 
$$ \langle \psi, \varphi\rangle := \sum_{x\in S} \psi(x)\varphi(x), $$
whenever the sum is absolutely convergent. 
\begin{lem}
\label{lem:<xi,phi>}Assume that 
$S=\Lambda_{n}$.
Let $(\xi_{0},\eta_{0})\in E_{fin}\times E_{fin}$. If $\varphi:S\rightarrow\mathbb{R}_{+}$,
 then 
\begin{equation}
\left\langle \xi_{t},\varphi\right\rangle =\left\langle \xi_{0},P_{t}\varphi\right\rangle +N_{t}^{\xi}(t,\varphi),\quad\left\langle \eta_{t},\varphi\right\rangle =\left\langle \eta_{0},P_{t}\varphi\right\rangle +N_{t}^{\eta}(t,\varphi),\label{eq:<xi,phi>}
\end{equation}
where
\begin{align}
N_{s}^{\xi}(t,\varphi)  = & \sum_{x\in S}\left(\sum_{y\neq x}\left\{ \intop_{0}^{s}\intop_{\mathbb{R}_{+}}P_{t-r}\varphi(x)1_{\xi_{r-}(y)\geq u}N_{y,x}^{RW_{\xi}}(drdu)\right.\right.\label{eq:N-t-phi-xi}\\
 &   \left.\left.-\intop_{0}^{s}\intop_{\mathbb{R}_{+}}P_{t-r}\varphi(x)1_{\xi_{r-}(x)\geq u}N_{x,y}^{RW_{\xi}}(drdu)\right\} -\intop_{0}^{s}\xi_{r}Q(x)P_{t-r}\varphi(x)dr\right)\nonumber \\
 &   +\sum_{x\in S}\sum_{k\geq0}(k-1)\intop_{0}^{s}\intop_{\mathbb{R}_{+}}P_{t-r}\varphi(x)1_{\{\gamma\eta_{r-}(x)\xi_{r-}(x)\geq u\}}N_{x,k}^{br_{\xi}}(drdu),\,\,\, s\leq t,\nonumber 
\end{align}
 and
\begin{eqnarray*}
N_{s}^{\eta}(t,\varphi) & = & \sum_{x\in S}\left(\sum_{y\neq x}\left\{ \intop_{0}^{s}\intop_{\mathbb{R}_{+}}P_{t-r}\varphi(x)1_{\eta_{r-}(y)\geq u}N_{y,x}^{RW_{\eta}}(drdu)\right.\right.\\
 &  & \left.\left.-\intop_{0}^{s}\intop_{\mathbb{R}_{+}}P_{t-r}\varphi(x)1_{\eta_{r-}(x)\geq u}N_{x,y}^{RW_{\eta}}(drdu)\right\} -\intop_{0}^{s}\eta_{r}Q(x)P_{t-r}\varphi(x)dr\right)\\
 &  & +\sum_{x\in S}\sum_{k\geq0}(
k-1)\intop_{0}^{s}\intop_{\mathbb{R}_{+}}P_{t-r}\varphi(x)1_{\{\gamma\eta_{r-}(x)\xi_{r-}(x)\geq u\}}N_{x,k}^{br_{\eta}}(drdu),\,\,\, s\leq t
\end{eqnarray*}
 are orthogonal  square-integrable $\mathcal{F}_{s}$-martingales
on $s\in[0,t]$ (the series converge in $L^{2}$ uniformly in $s\leq t$)
with conditional quadratic variations  
\begin{align}
\left\langle N_{\cdot}^{\xi}(t,\varphi)\right\rangle _{s}  = & \kappa\sum_{y\in S}\intop_{0}^{s}\xi_{r-}(y)\mathbb{E}\left[(P_{t-r}\varphi(Z+y)-P_{t-r}\varphi(y))^{2}\right]dr+\nonumber \\
 & +  \sigma^{2}\gamma\left(\sum_{x\in S}\intop_{0}^{s}\left(P_{t-r}\varphi(x)\right)^{2}\xi_{r-}(x)\eta_{r-}(x)dr\right),\label{eq:<N--t-phi-xi>}
\end{align}
and
\begin{eqnarray*}
\left\langle N_{\cdot}^{\eta}(t,\varphi)\right\rangle _{s} & = & \kappa\sum_{y\in S}\intop_{0}^{s}\eta_{r-}(y)\mathbb{E}\left[(P_{t-r}\varphi(Z+y)-P_{t-r}\varphi(y))^{2}\right]dr+\\
 & + & \sigma^{2}\gamma\left(\sum_{x\in S}\intop_{0}^{s}\left(P_{t-r}\varphi(x)\right)^{2}\xi_{r-}(x)\eta_{r-}(x)dr\right).
\end{eqnarray*}
Here $Z$ is the random variable distributed as a jump of the nearest
neighbor random walk.\end{lem}
\begin{proof}
The proof goes through application of Lemma~\ref{lem:Lemma} and  It\^o's formula to functions in the form of $f(s,\xi_s,\eta_s)=\left\langle \xi_{s},P_{t-s}\varphi\right\rangle$ 
and $f(s,\xi_s,\eta_s)=\left\langle \eta_{s},P_{t-s}\varphi\right\rangle$. The proof is pretty standard and we leave details to the enthusiastic reader.  

In the end the orthogonality of the martingales $N_{\cdot}^{\xi}(t,\varphi)$ and
$N_{\cdot}^{\eta}(t,\psi)$ follows from independence of driving Poisson
point processes.

\end{proof}

\begin{cor} 
\label{cor:E(xi*eta)=00003DE(xi)E(eta)}
Assume 
 $S=\Lambda_{n}$.
Let $(\xi_{0},\eta_{0})\in E_{fin}\times E_{fin}$.
If $\varphi,\psi:S\rightarrow\mathbb{R}_{+}$,
then
\begin{equation}
\mathbb{E}\left(\left\langle \xi_{t},\varphi\right\rangle \right)=\left\langle \xi_{0},P_{t}\varphi\right\rangle ,\,\,\,\,\,\mathbb{E}\left(\left\langle \eta_{t},\psi\right\rangle \right)=\left\langle \eta_{0},P_{t}\psi\right\rangle \label{eq:E<xi,phi>}, \; \forall t\geq 0, 
\end{equation}
and
\begin{equation}
\mathbb{E}\left(\left\langle \xi_{t},\varphi\right\rangle \left\langle \eta_{t},\psi\right\rangle \right)=\left\langle \xi_{0},P_{t}\varphi\right\rangle \left\langle \eta_{0},P_{t}\psi\right\rangle, \; \forall t\geq 0.\label{eq:E(<xi,phi><eta,psi>)}
\end{equation}
\end{cor}
\begin{proof}
Since $\Lambda_n$ is finite, \eqref{eq:E<xi,phi>}
follows immediately from Lemma~\ref{lem:<xi,phi>}.

As for \eqref{eq:E(<xi,phi><eta,psi>)}, recalling again that $\Lambda_n$ is finite,
 from Lemma \ref{lem:<xi,phi>} we get
\begin{eqnarray*}
\mathbb{E}\left(\left\langle \xi_{t},\varphi\right\rangle \left\langle \eta_{t},\psi\right\rangle \right) & = & \mathbb{E}\left(\left[\left\langle \xi_{0},P_{t}\varphi\right\rangle +N_{t}^{\xi}(t,\varphi)\right]\left[\left\langle \eta_{0},P_{t}\psi\right\rangle +N_{t}^{\eta}(t,\psi)\right]\right)\\
 & = & \left\langle \xi_{0},P_{t}\varphi\right\rangle \left\langle \eta_{0},P_{t}\psi\right\rangle +\left\langle \eta_{0},P_{t}\psi\right\rangle \mathbb{E}\left(N_{t}^{\xi}(t,\varphi)\right)\\
 &  & +\left\langle \xi_{0},P_{t}\varphi\right\rangle \mathbb{E}\left(N_{t}^{\eta}(t,\psi)\right)+\mathbb{E}\left(N_{t}^{\xi}(t,\varphi)N_{t}^{\eta}(t,\psi)\right)\\
 & = & \left\langle \xi_{0},P_{t}\varphi\right\rangle \left\langle \eta_{0},P_{t}\psi\right\rangle ,
\end{eqnarray*}
where the second and the third terms on the right-hand side equal
to zero since $N_{t}^{\xi}(t,\varphi)$ and $N_{t}^{\eta}(t,\psi)$ are
martingales, and the last term vanishes because of the orthogonality
of $N_{t}^{\xi}(t,\varphi)$ and $N_{t}^{\eta}(t,\psi)$. 
\end{proof}
Let $B\subseteq S$ be an arbitrary finite (bounded) subset, and let
$\left|B\right|$ denote a number of sites in $B$.

Now we are ready to compute the expected value and variance of a number
of particles (each population separately) at a site $x\in S$ and
at set $B$.
\begin{cor}
\label{lem:first-moment}Assume 
$S=\Lambda_{n}$.
Let $\xi_{0}(x)\equiv v,\,\,\eta_{0}(x)\equiv u\,\,\,\forall x\in S$, where  $(v,u)\in \mathbb{N}_{0}^2$. 
Then,
\[
\mathbb{E}(\xi_{t}(x))=v,\,\,\,\,\mathbb{E}(\eta_{t}(x))=u,  \,\,\,\,\mathrm{and}\,\,\,\,  \mathbb{E}(\xi_{t}(x)\eta_{t}(y))=vu\,\,\,\forall x,y\in S,\,\forall t\geq0.
\]
Moreover, for any finite $B\subset S$,
\[
\mathbb{E}(\xi_{t}(B))=v\left|B\right|,\,\,\,\,\mathrm{and}\,\,\,\,\mathbb{E}(\eta_{t}(B))=u\left|B\right|,\,\,\,\forall t>0.
\]
\end{cor}
\begin{proof}
The result follows easily from  
Corollary~\ref{cor:E(xi*eta)=00003DE(xi)E(eta)}. 
\end{proof}
Before we treat the second moments of $\xi_{t}(x)$ and $\eta_{t}(x)$,
let us prove a simple technical lemma. Recall that $g_{t}(\cdot,\cdot)$
is the Green function defined in \eqref{eq:green-f-def} (for the
nearest neighbor random walk on 
 $S=\Lambda_{n}$).
Note that 
whenever the motion process is the nearest neighbor random walk, we
have $g_{t}(x,y)=g_{t}(x-y)$ (with certain abuse of notation).
\begin{lem}
\label{lem:g0-ge1}Assume 
 $S=\Lambda_{n}$,
and the motion process is a nearest neighbor random walk on $S$.
For every $x\in S$: 
\[
g_{t}(x)-\frac{1}{2d}\sum_{i=1}^{d}\left[g_{t}(x+e_{i})+g_{t}(x-e_{i})\right]=\frac{1}{\kappa}(\delta_{x,0}-p_{t}(x)),\,\,\forall t\geq0,
\]
where $\delta_{x,y}=1$ iff $x=y$ and $\delta_{x,y}=0$ otherwise.\end{lem}
\begin{proof}
The proof follows a standard procedure, using the evolution equation for the transition densities of a continuous time nearest neighbor random walk. As we could not find an exact reference, we have included the derivation here for completeness.

The evolution of the transition probabilities of a continuous-time Markov
chain is governed by the first-order differential equation
\[
\frac{\partial}{\partial t}p_{t}(x)=Qp_{t}(x),\,\,\,\forall t\geq0,x\in S.
\]
In our case $Q=(q_{xy})_{x,y\in S}$ where 
\[
q_{xx}=-\kappa,\,\, q_{xy}=\frac{1}{2d}\kappa\,\,\mathrm{for}\,\, x\in S,\, y=x\pm e_{i},\, i=1,...,d\,\,\mathrm{and}\,\, q_{xy}=0\,\,\mathrm{otherwise}.
\]
Therefore, we have:
\begin{eqnarray*}
\frac{\partial}{\partial t}p_{t}(x) 
 & = & \sum_{y\in S}q_{xy}p_{t}(y)=q_{xx}(p_{t}(x)-\frac{1}{2d}\sum_{i=1}^{d}\left(p_{t}(x+e_{1})+p_{t}(x-e_{1})\right))\\
 & = & -\kappa\left(p_{t}(x)-\frac{1}{2d}\sum_{i=1}^{d}\left[p_{t}(x+e_{i})+p_{t}(x-e_{i})\right]\right).
\end{eqnarray*}
Now, integrate both sides over the interval $[0,t]$, gives
\begin{eqnarray*}
p_{t}(x)-p_{0}(x) 
 & = & -\kappa\intop_{0}^{t}(p_{s}(x)-\frac{1}{2d}\sum_{i=1}^{d}\left[p_{s}(x+e_{i})+p_{s}(x-e_{i})\right])ds\\
 & = & -\kappa(g_{t}(x)-\frac{1}{2d}\sum_{i=1}^{d}\left[g_{t}(x+e_{i})+g_{t}(x-e_{i})\right]).
\end{eqnarray*}
The result follows immediately once we recall that $p_0(x)=\delta_{x,0}$. 

\end{proof}
Now we are ready to handle the second moments of $\xi_{t}(x)$ and
$\eta_{t}(x)$.
\begin{lem}
\label{lem:Second-moment}Assume 
 $S=\Lambda_{n}$.
Let $\xi_{0}(x)\equiv v,\,\,\eta_{0}(x)\equiv u$ for all $x\in S$, where $(v,u)\in \mathbb{N}_{0}^2$. 
Then, for all $t\geq0$,
\begin{equation}
\mathbb{E}(\xi_{t}(x)^{2})=v^{2}+\frac{1}{2}\sigma^{2}\gamma uvg_{2t}(0)+v(1-p_{2t}(0)),\label{eq:*}
\end{equation}
and
\begin{equation}
\mathbb{E}(\eta_{t}(x)^{2})=u^{2}+\frac{1}{2}\sigma^{2}\gamma uvg_{2t}(0)+u(1-p_{2t}(0)).\label{eq:**}
\end{equation}
\end{lem}
\begin{proof}
We will prove only \eqref{eq:*}, since the proof of \eqref{eq:**}
is the same. Again we use the representation of the process from Lemma
\ref{lem:<xi,phi>}, with $\varphi\left(\cdot\right)=\delta_{x}\left(\cdot\right)$
and use notation $\phi_{r}(\cdot)=p_{t-r}\delta_{x}(\cdot)=p_{t-r}(\cdot-x).$

For $0\leq s\leq t$ denote 
\[
N_{s}^{\xi}(y)=N_{s}^{\xi}(t,y)=N_{s}^{\xi}(t,\delta_{y}).
\]

\begin{eqnarray*}
\mathbb{E}(\xi_{t}(x)^{2})
 & = & \left(P_{t}\xi_{0}(x)\right)^{2}+\mathbb{E}\left(\left(N_{t}^{\xi}(x)\right)^{2}\right)\\
 & = & v^{2}+\mathbb{E}\left(\left\langle N_{\cdot}^{\xi}(x)\right\rangle _{t}\right)\\
 & = & v^{2}+v\kappa\sum_{y\in S}\intop_{0}^{t}\sum_{z\in S}(\phi_{r}(z)-\phi_{r}(y))^{2}p{}_{y,z}dr\\
 &  & +\sigma^{2}\gamma uv\sum_{y\in S}\intop_{0}^{t}\phi_{r}(y)^{2}dr\\
 & =: & v^{2}+v\kappa J_{1}(t)+\sigma^{2}\gamma uv J_{2}(t),\,\,\,\,\,\,\,\,\,\,\,\,\,\,\,\,\,\,\,\,\,\,\,\,\,\,\,\,\,\,\,\,\,\,\,\, t\geq0
\end{eqnarray*}
where in the third equality we used again Lemma \ref{lem:<xi,phi>},  the Fubini theorem, and  Corollary~\ref{cor:E(xi*eta)=00003DE(xi)E(eta)} which implies 
$\mathbb{E}(\xi_{r-}(y)\eta_{r-}(y))=P_{t}\xi_{0}(y)P_{t}\eta_{0}(y)=vu$, $\mathbb{E}(\xi_{r-}(y))=P_{t}\xi_{0}(y)=v.$
 Now we will compute each term separately.

First, let us evaluate $J_{2}(t)$. Recall that $\phi_{r}(y)=p_{t-r}(y-x)$.
Then we have
\begin{align}
J_{2}(t)  = & 
\intop_{0}^{t}\sum_{y\in S}p_{t-r}(y-x)^{2}dr
=\intop_{0}^{t}p_{2(t-r)}(0)dr
=\frac{1}{2}\intop_{0}^{2t}p_{\tau}(0)d\tau=0.5g_{2t}(0),\label{eq:J2}
\end{align}
where $p_{s}(x,x)=p_{s}(0,0)=p_{s}(0),$ for all $x\in S$.

Now we will handle $J_{1}(t)$:
\begin{align}
J_{1}(t) & =  \sum_{y\in S}\intop_{0}^{t}\sum_{z\in S}(\phi_{r}(z)-\phi_{r}(y))^{2}p{}_{y,z}dr\label{eq:p_s^2_calc}\\
 & =  \sum_{y\in S}\sum_{z\in S}\intop_{0}^{t}p_{t-r}(z-x)^{2}p_{y,z}dr+\sum_{y\in S}\sum_{z\in S}\intop_{0}^{t}p_{t-r}(y-x)^{2}p_{y,z}dr\nonumber \\
 &   -2\sum_{y\in S}\sum_{z\in S}\intop_{0}^{t}p_{t-r}(z-x)p_{t-r}(y-x)p_{y,z}dr.\nonumber 
\end{align}
We will treat each of the three terms above separately.
For the first term we have
\begin{align}
\sum_{y\in S}\sum_{z\in S}\intop_{0}^{t}p_{t-r}(z,x)^{2}p_{y,z}dr 
 & =  \sum_{z\in S}\intop_{0}^{t}p_{t-r}(z,x)^{2}dr
= 0.5g_{2t}(0)\label{eq:J11}
\end{align}
where the last equality follows as in \eqref{eq:J2}.

Similarly we get

\begin{align}
\sum_{y\in S}\sum_{z\in S}\intop_{0}^{t}p_{t-r}(y,x)^{2}p_{y,z}dr 
&=0.5g_{2t}(0).\label{eq:J12}
\end{align}
 Finally it is easy to obtain
\begin{multline}
\sum_{y\in S}\sum_{z\in S}\intop_{0}^{t}p_{t-r}(z,x)p_{t-r}(y,x)p_{y,z}dr\\
\begin{aligned}
&=0.5\frac{1}{2d}\sum_{i=1}^{d}\left[g_{2t}(e_{i})+g_{2t}(-e_{i})\right]=\frac{1}{2d}\sum_{i=1}^{d}g_{2t}(e_{i}).
\end{aligned}
\label{eq:J13}
\end{multline}
By putting \eqref{eq:J11}, \eqref{eq:J12}, \eqref{eq:J13} and \eqref{eq:J2}
together we have 
\begin{equation}
\mathbb{E}(\xi_{t}(x)^{2})=v^{2}+\sigma^{2}\gamma uv\frac{1}{2}g_{2s}(0)+v\kappa(g_{2t}(0)-\frac{1}{2d}\sum_{i=1}^{d}g_{2t}(e_{i})),\,\, t\geq0,\,\, x\in S.\label{eq:E(xi^2)}
\end{equation}
Now use \eqref{eq:E(xi^2)} and Lemma \ref{lem:g0-ge1} with $x=y$
to get 
\[
\mathbb{E}(\xi_{t}(x)^{2})=v^{2}+\frac{1}{2}\sigma^{2}\gamma uvg_{2t}(0)+v(1-p_{2t}(0)),\,\,\,\forall t\geq0,x\in S.
\]

\end{proof}
We will also need to evaluate $\mathbb{E}(\xi_{t}(x)\xi_{t}(y))$
for $x\neq y$. To this end we will prove the following lemma. 
\begin{lem}
\label{lem:mixed moment}Assume 
 $S=\Lambda_{n}$.
Let $\xi_{0}(x)\equiv v,\,\,\eta_{0}(x)\equiv u\,\,\,\forall x\in S$, where  $(v,u)\in \mathbb{N}_{0}^2$. 
Let $x\neq y$, then
\[
\mathbb{E}(\xi_{t}(x)\xi_{t}(y))=v^{2}-vp_{2t}(x-y)+\frac{1}{2}\sigma^{2}\gamma uvg_{2t}(x-y),\,\,\,\forall t\geq0,
\]
and
\[
\mathbb{E}(\eta_{t}(x)\eta_{t}(y))=u^{2}-up_{2t}(x-y)+\frac{1}{2}\sigma^{2}\gamma uvg_{2t}(x-y),\,\,\,\forall t\geq0.
\]
\end{lem}
\begin{proof} 
The proof goes along similar lines as the proof of Lemma~\ref{lem:Second-moment} and thus is omitted. 

\end{proof}

\section{Proof of Theorem \ref{thm:MainResult-1}\label{chap:Finite-System-Scheme}}

Let $\left(\xi_{t}^{n},\eta_{t}^{n}\right)$ be a pair of processes
solving \eqref{eq:main_equation} with site space $S=\Lambda_{n}$,
and $N_{x,y}^{RW_{\xi}}$, $N_{x,y}^{RW_{\eta}}$ being Poisson point
processes with intensity measure $\mbox{\ensuremath{q^{n}(x,y)ds\otimes du}}$,
$q^{n}$ is defined by \eqref{eq:transition_kernel}. $\left\{ p_{x,y}^{n}\right\} _{x,y\in\Lambda_{n}}$
are the transition jump probabilities of the underlying random walk, and $\left\{ P_{t}^{n}\right\} _{t\geq0}$
is the associated semigroup. In what follows we assume $d\geq 3$. 

Fix $(\theta_{1},\theta_{2}) \in \mathbb{N}_{0}^2$. Assume the following initial conditions
for $\left(\xi_{t}^{n},\eta_{t}^{n}\right)$ : 
\[
\xi_{0}^{n}(x)=\theta_{1}\,\,\,\eta_{0}^{n}(x)=\theta_{2}\,\,\,\forall x\in\Lambda_{n}.
\]
Set

\begin{equation}
\boldsymbol{\xi}_{t}^{n}=\sum_{j\in\Lambda_{n}}\xi_{j}^{n}(t),\,\,\,\,\,\,\boldsymbol{\eta}_{t}^{n}=\sum_{j\in\Lambda_{n}}\eta_{j}^{n}(t).\label{eq:xi_N_def}
\end{equation}
We define the following time change:
\[
\beta_{n}(t)=\left|\Lambda_{n}\right|t,\,\,\, t\geq0.
\]
Theorem \ref{thm:MainResult-1} identifies the limiting distribution
of
\[
\frac{1}{\left|\Lambda_{n}\right|}\left(\boldsymbol{\xi}_{\beta_{n}(t)}^{n},\boldsymbol{\eta}_{\beta_{n}(t)}^{n}\right),
\]
as $n\rightarrow\infty$, for $t\in [0,1]$.

In Section  \ref{sec:Introduction} we defined a
system of Dawson-Perkins processes $(U_{t}^{n},V_{t}^{n})_{t\geq0}$
on $\Lambda_{n}$, that solves \eqref{eq:DP-Lambda_n}. Recall that
\[
\mathbf{U}_{t}^{n}=\sum_{i\in\Lambda_{n}}u_{t}^{n}(i),\,\,\,\,\,\,\mathbf{V}_{t}^{n}=\sum_{i\in\Lambda_{n}}v_{t}^{n}(i).
\]

The limiting behavior of $(\mathbf{U}_{t}^{n},\mathbf{V}_{t}^{n})_{t\geq0}$
was studied in \cite{CDG04}, we stated the result in Theorem~\ref{thm:CDG-1}.

Theorem \ref{thm:MainResult-1} claims that the limiting behavior
of $\frac{1}{\left|\Lambda_{n}\right|}\left(\boldsymbol{\xi}_{\beta_{n}(t)}^{n},\boldsymbol{\eta}_{\beta_{n}(t)}^{n}\right)$
is similar to $\frac{1}{\left|\Lambda_{n}\right|}\left(\mathbf{U}_{\beta_{n}(t)}^{n},\mathbf{V}_{\beta_{n}(t)}^{n}\right)$
for $t\in [0,1]$. As we have mentioned above, contrary to Dawson-Perkins processes solving equation~\eqref{eq:DP-Lambda_n}, the useful self-duality property does not hold for our branching particle model. However we use 
the so called approximating duality technique that allows us to prove Theorem \ref{thm:MainResult-1}.  

In what follows we will use a periodic sum on $\Lambda_{n}$:
for $x,y\in\Lambda_{n}$ we have $x+y=(x+y)\mod\Lambda_{n}\in\Lambda_{n}$.

Next proposition is crucial for the proof of Theorem \ref{thm:MainResult-1}. 
\begin{prop}
\label{prop:Laplace-Fourier}Let $(X_{t},Y_{t})_{t\geq0}$ be the
solution to \eqref{eq:XY-1}. Then for all $\mbox{\ensuremath{a,b\geq0}}$,
\begin{align*}
\lim_{n\rightarrow\infty}\mathbb{E}\left(e^{-\frac{1}{\left|\Lambda_{n}\right|}(\boldsymbol{\xi}_{\beta_{n}(t)}^{n}+\boldsymbol{\eta}_{\beta_{n}(t)}^{n})(a+b)-i\frac{1}{\left|\Lambda_{n}\right|}(\boldsymbol{\xi}_{\beta_{n}(t)}^{n}-\boldsymbol{\eta}_{\beta_{n}(t)}^{n})(a-b)}\right)\\
=\mathbb{E}\left(e^{-(X_{t}+Y_{t})(a+b)-i(X_{t}-Y_{t})(a-b)}\right),
\end{align*}
 for $t\in [0,1]$. 
\end{prop}
\noindent {\bf Proof of Theorem \ref{thm:MainResult-1}.}
By easy adaptation of Lemma 2.5 of \cite{Mytnik98} one gets that
the mixed Laplace-Fourier transform 
\[
\mathbb{E}\left(e^{-(X+Y)(a+b)-i(X-Y)(a-b)}\right),\,\,\, a,b\geq0,
\]
determines the distribution of non-negative two-dimentional random
variables $(X,Y)$. Therefore, Theorem \ref{thm:MainResult-1} follows
easily from Proposition~\ref{prop:Laplace-Fourier} and properties of weak convergence.

\qed 


The rest of the section is organized as follows. Section~\ref{sec:dual_thm.26}  is devoted to the proof of Proposition~\ref{prop:Laplace-Fourier}, and  the proof of one of the technical propositions is deferred to Section~\ref{sec:techn_prop}.

\subsection{Proof of Proposition~\ref{prop:Laplace-Fourier}}\label{sec:dual_thm.26}


In what follows fix $T\in (0,1]$. Let $(\xi_{t}^{n},\eta_{t}^{n})_{t\geq0}$
be a mutually catalytic branching random walk from Theorem \ref{thm:MainResult-1}
(we will refer to it as the ``discrete process''). In the proof
of the proposition we will use the duality technique introduced in
\cite{Mytnik98}. To this end we will need the following Dawson-Perkins
processes:
\begin{itemize}
\item Let $(u_{t}^{n},v_{t}^{n})_{t\geq0}$ be a solution to \eqref{eq:DP-Lambda_n},
with $Q^n$ being the $Q$-matrix of the nearest neighbor random walk on $\Lambda_n$,
and with some initial conditions $\left(u_{0},v_{0}\right)$.
\item For arbitrary $a,b\geq 0$, the sequence $(\tilde{u}_{t}^{n},\tilde{v}_{t}^{n})_{t\geq0}$
solving \eqref{eq:DP-Lambda_n} with initial conditions 
\begin{equation}
\tilde{u}_{0}^{n}(x)=\frac{a}{\left|\Lambda_{n}\right|},\,\,\,\tilde{v}_{0}^{n}(x)=\frac{b}{\left|\Lambda_{n}\right|}\,\,\,\mathrm{for\,\, every\,\,}x\in\Lambda_{n}.\label{eq:u_n,v_n-init-cond}
\end{equation}
 
\end{itemize}
In what follows we assume that $\left(u_{t}^{n},v_{t}^{n}\right)_{t\geq0},\,\left(\tilde{u}_{t}^{n},\tilde{v}_{t}^{n}\right)_{t\geq0}$
and $\left(\xi_{t}^{n},\eta_{t}^{n}\right)_{t\geq0}$ are independent.
Now let us describe the state spaces for the processes involved in
this section. 

Similarly to $E_{fin}$ define
$E_{fin}^{n}=\left\{ f:\Lambda_{n}\longrightarrow\mathbb{N}_{0}\right\}$ , and $E_{fin,con}^{n}= E_{fin}^{n}\times E_{fin}^{n}$. Clearly, since $\Lambda_{n}$ is compact, the 
$L^1$ norm  of functions in $E_{fin}^{n}$ is finite. Also,  define $\widetilde{E}_{fin}^{n}=\left\{ f:\Lambda_{n}\longrightarrow\mathbb{R}_{+}\right\}$ , and $\widetilde{E}_{fin,con}^{n}= \widetilde{E}_{fin}^{n}\times \widetilde{E}_{fin}^{n}$.

First, by Theorem~\ref{lem:Lemma-1}, the process $\left(\xi_{t}^{n},\eta_{t}^{n}\right)$
that solves \eqref{eq:main_equation} with initial conditions $\left(\xi_{0}^{n},\eta_{0}^{n}\right)=\bar{\boldsymbol{\theta}}$
is $E_{fin}^{n}\times E_{fin}^{n}$-valued process. By our definition
\eqref{eq:u_n,v_n-init-cond}, $\left(\tilde{u}_{0}^{n},\tilde{v}_{0}^{n}\right)\in \widetilde{E}_{fin,con}^{n}$.
Moreover, by simple adaptation of the proof of Theorem 2.2(d) in \cite{DawsonPerkins98}
to our state space $\Lambda_{n}$, we get 
\[
\left(\tilde{u}_{t}^{n},\tilde{v}_{t}^{n}\right)\in \widetilde{E}_{fin,con},\,\,\forall t\geq0.
\]

For $(\varphi,\psi,\tilde\varphi,\tilde\psi)\in\mathbb{R}_{+}^{\Lambda_{n}}\times\mathbb{R}_{+}^{\Lambda_{n}}\times\mathbb{R}_{+}^{\Lambda_{n}}\times\mathbb{R}_{+}^{\Lambda_{n}}$
define
\[
H\left(\varphi,\psi,\tilde\varphi,\tilde\psi\right)=e^{-\left\langle \varphi+\psi,\tilde\varphi+\tilde\psi\right\rangle -i\left\langle \varphi-\psi,\tilde\varphi-\tilde\psi\right\rangle },
\]
and 
\begin{eqnarray*}
F_{t,s}^{n} & = & \mathbb{E}\left[H\left(\xi_{t}^{n},\eta_{t}^{n},\tilde{u}_{s}^{n},\tilde{v}_{s}^{n}\right)\right]\\
 & = & \mathbb{E}\left[e^{-\left\langle \xi_{t}^{n}+\eta_{t}^{n},\tilde{u}_{s}^{n}+\tilde{v}_{s}^{n}\right\rangle -i\left\langle \xi_{t}^{n}-\eta_{t}^{n},\tilde{u}_{s}^{n}-\tilde{v}_{s}^{n}\right\rangle }\right],
\end{eqnarray*}
for $0\leq s,t\leq\beta_{n}(T)$.

Let us recall the self-duality lemma from \cite{CDG04} (Lemma 4.1
in \cite{CDG04}).
\begin{lem}
\label{lem:Self-duality}Let $\left(u_{0},v_{0}\right),\left(\tilde{u}_{0},\tilde{v}_{0}\right)\in \widetilde{E}_{fin,con}^{n}$,
where $\left(u_{t},v_{t}\right)_{t\geq0},\left(\tilde{u}_{t},\tilde{v}_{t}\right)_{t\geq0}$
are independent solutions of \eqref{eq:DP-Lambda_n}. Then
\[
\mathbb{E}\left(H(u_{t},v_{t},\tilde{u}_{0},\tilde{v}_{0})\right)=\mathbb{E}\left(H(u_{0},v_{0},\tilde{u}_{t},\tilde{v}_{t})\right).
\]
 \end{lem}
\begin{rem}
In \cite{CDG04} the above lemma is proved for more general state
spaces and initial conditions. The conditions in Lemma 4.1 in \cite{CDG04}
hold trivially in our case.
\end{rem}
Then we have the following proposition.
\begin{prop}
\label{lem:limFn}For any $(\theta_{1},\theta_{2})\in \mathbb{N}_{0}^2$, $a,b\geq0$,
\begin{equation}
\lim_{n\rightarrow\infty}\mathbb{E}\left[F_{\beta_{n}(T),0}^{n}\right]=\lim_{n\rightarrow\infty}\mathbb{E}\left[F_{0,\beta_{n}(T)}^{n}\right].\label{eq:Fn}
\end{equation}
 \end{prop}
\begin{proof}
Postponed till the end of this section. It is proved via a series
of other results.
\end{proof}
Given Proposition \ref{lem:limFn}, it is easy to complete
\begin{proof}[Proof of Proposition \ref{prop:Laplace-Fourier}.]
\emph{ }Fix arbitrary $\theta_{1},\theta_{2}\geq 0$ and $a,b\geq0$.
For any $n\geq1$, let $\left(u_{t}^{n},v_{t}^{n}\right)_{t\geq0}$
be the solution to \eqref{eq:DP-Lambda_n} with $Q^{n}$ being a $Q$-matrix
of the nearest neighbor random walk on $\Lambda_{n}$, and initial
conditions $\left(u_{0}^{n},v_{0}^{n}\right)=\left(\xi_{0}^{n},\eta_{0}^{n}\right)=\boldsymbol{\bar{\theta}}$.
Recall that
\[
\boldsymbol{U}_{t}^{n}=\sum_{x\in\Lambda_{n}}u_{t}^{n}(x),\,\,\,\boldsymbol{V}_{t}^{n}=\sum_{x\in\Lambda_{n}}v_{t}^{n}(x).
\]
Note that
\begin{align}
\lim_{n\rightarrow\infty}\mathbb{E}\left[F_{0,\beta_{n}(T)}^{n}\right]  = & \lim_{n\rightarrow\infty}\mathbb{E}\left(e^{-\left\langle \xi_{0}^{n}+\eta_{0}^{n},\tilde{u}_{\beta_{n}(T)}^{n}+\tilde{v}_{\beta_{n}(T)}^{n}\right\rangle -i\left\langle \xi_{0}^{n}-\eta_{0}^{n},\tilde{u}_{\beta_{n}(T)}^{n}-\tilde{v}_{\beta_{n}(T)}^{n}\right\rangle }\right)\nonumber \\
  = & \lim_{n\rightarrow\infty}\mathbb{E}\left(e^{-\left(\mathbf{U}_{\beta_{n}(T)}^{n}+\mathbf{V}_{\beta_{n}(T)}^{n}\right)\frac{1}{\left|\Lambda_{n}\right|}(a+b)-i\left(\mathbf{U}_{\beta_{n}(T)}^{n}-\mathbf{V}_{\beta_{n}(T)}^{n}\right)\frac{1}{\left|\Lambda_{n}\right|}(a-b)}\right)\nonumber \\
  = & \mathbb{E}\left(e^{-(X_{T}+Y_{T})(a+b)-i(X_{T}-Y_{T})(a-b)}\right),\label{eq:limF^n_0_beta}
\end{align}
where the second equality follows by a self-duality relation in Lemma
\ref{lem:Self-duality}, and the third equality follows by Theorem
\ref{thm:CDG-1}. This means that 
\begin{align}
\lim_{n\rightarrow\infty}\mathbb{E}\left(e^{-(\boldsymbol{\xi}_{\beta(T)}^{n}+\boldsymbol{\eta}_{\beta_{n}(T)}^{n})\frac{1}{\left|\Lambda_{n}\right|}(a+b)-i(\boldsymbol{\xi}_{\beta_{n}(T)}^{n}-\boldsymbol{\eta}_{\beta_{n}(T)}^{n})\frac{1}{\left|\Lambda_{n}\right|}(a-b)}\right)\label{eq:lim_dist}\\
\begin{aligned}= & \lim_{n\rightarrow\infty}\mathbb{E}\left[F_{\beta_{n}(T),0}^{n}\right]=\lim_{n\rightarrow\infty}\mathbb{E}\left[F_{0,\beta_{n}(T)}^{n}\right]\\
= & \mathbb{E}\left(e^{-(X_{T}+Y_{T})(a+b)-i(X_{T}-Y_{T})(a-b)}\right),
\end{aligned}
\end{align}
where the second equality follows by Proposition \ref{lem:limFn},
and the last equality follows by \eqref{eq:limF^n_0_beta}. This finishes
the proof of Proposition \ref{prop:Laplace-Fourier}.
\end{proof}
To prove Proposition \ref{lem:limFn} we will need other results.
First we need Lemma 4.10  from \cite{EthierKurtz}.
\begin{lem}[ Lemma 4.10  in \cite{EthierKurtz}]
\label{thm:KURTZ-thm}Suppose a function $f(s,t)$ on $\left[0,\infty\right)\times\left[0,\infty\right)$
is absolutely continuous in $s$ for each fixed $t$ and absolutely
continuous in $t$ for each fixed $s$. Set $\left(f_{1},f_{2}\right)\equiv\nabla f$,
and assume that 
\begin{equation}
\intop_{0}^{T}\intop_{0}^{T}\left|f_{i}(s,t)\right|dsdt<\infty,\,\, i=1,2,\,\,\forall T>0.\label{eq:KURTZ-condition}
\end{equation}
Then for almost every $t\geq0$,
\begin{equation}
f\left(t,0\right)-f\left(0,t\right)=\intop_{0}^{t}\left(f_{1}\left(s,t-s\right)-f_{2}\left(s,t-s\right)\right)ds.\label{eq:KURTZ-result}
\end{equation}

\end{lem}
We will apply this lemma for the function $F_{r,s}^{n}=\mathbb{E}\left[H\left(\xi_{r}^{n},\eta_{r}^{n},\tilde{u}_{s}^{n},\tilde{v}_{s}^{n}\right)\right]$.
Then we will show that for $f(r,s)=F_{r,s}^{n}$ and $t=\beta_{n}(T)$,
the right-hand side of \eqref{eq:KURTZ-result} tends to $0$, as
$n\rightarrow\infty$.

In order to check the conditions in Lemma \ref{thm:KURTZ-thm} we
will need several lemmas. In the next two lemmas we will derive martingale
problems for processes $(\xi_{\cdot}^{n},\eta_{\cdot}^{n})$ and $(\tilde{u}_{\cdot}^{n},\tilde{v}_{\cdot}^{n})$.
Recall that $\left\{ p_{x,y}^{n}\right\} _{x,y\in\Lambda_{n}}$
are the transition jump probabilities of the underlying  nearest-neighbor random walk on $\Lambda_{n}$.
\begin{lem}
\label{lem:g-mart} For any $(\varphi,\psi)\in \widetilde{E}_{fin,con}^{n}$
define
\begin{multline}
g\left(\xi_{s}^{n},\eta_{s}^{n},\varphi,\psi\right)=H\left(\xi_{s}^{n},\eta_{s}^{n},\varphi,\psi\right)\left\{ \kappa\sum_{x,y\in\Lambda_{n}}\xi_{s}^{n}(x)\right.\\
\begin{aligned} & \left.\times p_{xy}^{n}\left[e^{-\varphi(y)-\psi(y)+\varphi(x)+\psi(x)-i\left(\varphi(y)-\psi(y)-\varphi(x)+\psi(x)\right)}-1\right]\right.\\
 & +\kappa\sum_{x,y\in\Lambda_{n}}\eta_{s}^{n}(x)p_{xy}^{n}\left[e^{-\varphi(y)-\psi(y)+\varphi(x)+\psi(x)+i\left(\varphi(y)-\psi(y)-\varphi(x)+\psi(x)\right)}-1\right]\\
 & +\gamma\sum_{x\in\Lambda_{n}}\xi_{s}^{n}(x)\eta_{s}^{n}(x)\sum_{k\geq0}\nu_{k}\left[e^{-(k-1)\left(\varphi(x)+\psi(x)+i\left(\varphi(x)-\psi(x)\right)\right)}-1\right]\\
 & \left.+\gamma\sum_{x\in\Lambda_{n}}\xi_{s}^{n}(x)\eta_{s}^{n}(x)\sum_{k\geq0}\nu_{k}\left[e^{-(k-1)\left(\varphi(x)+\psi(x)-i\left(\varphi(x)-\psi(x)\right)\right)}-1\right]\right\} ,\,\,\,\forall s\geq0.
\end{aligned}
\label{eq:g-def}
\end{multline}
 Then
\[
H\left(\xi_{t}^{n},\eta_{t}^{n},\varphi,\psi\right)-\intop_{0}^{t}g\left(\xi_{s}^{n},\eta_{s}^{n},\varphi,\psi\right)ds,\,\,\,\forall t\geq0.
\]
is an $\left\{ \mathcal{F}_{t}^{\xi,\eta}\right\} _{t\geq0}$-martingale.\end{lem}
\begin{proof}
The result is immediately by Lemma \ref{lem:Lemma}(d).
\end{proof}
Similar result holds for the Dawson-Perkins process.
\begin{lem}
\label{lem:h-mart}For any $(\varphi,\psi)\in E_{fin,con}^{n}$,
define
\begin{align}
h\left(\varphi,\psi,\tilde{u}_{s}^{n},\tilde{v}_{s}^{n}\right) = & H\left(\varphi,\psi,\tilde{u}_{s}^{n},\tilde{v}_{s}^{n}\right)\left\{- \sum_{x\in\Lambda_{n}}\tilde{u}_{s}^{n}Q^n(x)\left(\varphi(x)+\psi(x)+i\left(\varphi(x)-\psi(x)\right)\right)\right.\nonumber \\
   & -\sum_{x\in\Lambda_{n}}\tilde{v}_{s}^{n}Q^n(x)\left(\varphi(x)+\psi(x)-i\left(\varphi(x)-\psi(x)\right)\right)\nonumber \\
   & \left.+4\tilde{\gamma}\sum_{x\in\Lambda_{n}}\tilde{u}_{s}^{n}(x)\tilde{v}_{s}^{n}(x)\varphi(x)\psi(x)\right\} ,\,\,\,\forall s\geq0.\label{eq:h-def}
\end{align}
Then 
\[
H\left(\varphi,\psi,\tilde{u}_{t}^{n},\tilde{v}_{t}^{n}\right)-\intop_{0}^{t}h\left(\varphi,\psi,\tilde{u}_{s}^{n},\tilde{v}_{s}^{n}\right)ds,\,\,\, t\geq0,
\]
is an $\left\{ \mathcal{F}_{t}^{\tilde{u}^{n},\tilde{v}^{n}}\right\} _{t\geq0}$-martingale.\end{lem}
\begin{proof}
The result is immediate by Theorem 2.2(c)(iv) in \cite{DawsonPerkins98}, It$\mathrm{\hat{o}}$'s
lemma (Theorem II.5.1 in \cite{IkedaWatanabe}) and simple algebra. 
\end{proof}
\begin{lem}
\label{lem:Eg,Eh<inf}For any $t>0$,
\begin{equation}
\sup_{\begin{array}{c}
0\leq s\leq t\\
0\leq r\leq t
\end{array}}\mathbb{E}\left|h(\xi_{r}^{n},\eta_{r}^{n},\tilde{u}_{s}^{n},\tilde{v}_{s}^{n})\right|<\infty\label{eq:Eh<inf}
\end{equation}
and
\begin{equation}
\sup_{\begin{array}{c}
0\leq s\leq t\\
0\leq r\leq t
\end{array}}\mathbb{E}\left|g(\xi_{r}^{n},\eta_{r}^{n},\tilde{u}_{s}^{n},\tilde{v}_{s}^{n})\right|<\infty.\label{eq:Eg<inf}
\end{equation}
\end{lem}
\begin{proof}
\eqref{eq:Eh<inf} is verified in the proof of Theorem 2.4(b) in \cite{DawsonPerkins98}
. 

Now let us check \eqref{eq:Eg<inf}. First, by simple algebra it
is trivial to see that for any $z\in\mathbb{R}_{+}$ and $y\in\mathbb{R}$,
\begin{align}
\left|e^{-z+iy}-1\right| 
 & \leq 
\left(\left|z\right|+\left|y\right|\right).\label{eq:bound-e^-z}
\end{align}
Hence
\begin{multline*}
\sup_{\begin{array}{c}
0\leq s\leq t\\
0\leq r\leq t
\end{array}}\mathbb{E}\left|g(\xi_{r}^{n},\eta_{r}^{n},\tilde{u}_{s}^{n},\tilde{v}_{s}^{n})\right|\\
\begin{aligned}
\leq & \sup_{0\leq s,r\leq t}C\mathbb{E}\left\{ \kappa\sum_{x,y\in\Lambda_{n}}\xi_{r}^{n}(x)p_{xy}^{n}\left[\tilde{u}_{s}^{n}(y)+\tilde{v}_{s}^{n}(y)+\tilde{u}_{s}^{n}(x)+\tilde{v}_{s}^{n}(x)\right]\right.\\
 & +\kappa\sum_{x,y\in\Lambda_{n}}\eta_{r}^{n}(x)p_{xy}^{n}\left[\tilde{u}_{s}^{n}(y)+\tilde{v}_{s}^{n}(y)+\tilde{u}_{s}^{n}(x)+\tilde{v}_{s}^{n}(x)\right]\\
 & +\gamma\sum_{x\in\Lambda_{n}}\xi_{r}^{n}(x)\eta_{r}^{n}(x)\sum_{k\geq0}\nu_{k}\left|k-1\right|\left[\tilde{u}_{s}^{n}(x)+\tilde{v}_{s}^{n}(x)\right]\\
 & \left.+\gamma\sum_{x\in\Lambda_{n}}\xi_{r}^{n}(x)\eta_{r}^{n}(x)\sum_{k\geq0}\nu_{k}\left|k-1\right|\left[\tilde{u}_{s}^{n}(x)+\tilde{v}_{s}^{n}(x)\right]\right\} ,
\end{aligned}
\end{multline*}
where $C>0$ is a constant and the last inequality follows from \eqref{eq:bound-e^-z}.
Recall that, by Corollary~\ref{lem:first-moment}, $\mathbb{E}\left[\xi_{s}^{n}(x)\right]=\theta_{1},\,\,\mathbb{E}\left[\eta_{s}^{n}(x)\right]=\theta_{2}$
and,
$\mathbb{E}\left[\xi_{s}^{n}(x)\eta_{s}^{n}(x)\right]=\theta_{1}\theta_{2}$.
By Theorem 2.2b(iii) in \cite{DawsonPerkins98}, 
\begin{eqnarray*}
\mathbb{E}\left[\left\langle \tilde{u}_{s}^{n},1\right\rangle \right] & = & a<\infty,\,\,\mathbb{E}\left[\left\langle \tilde{v}_{s}^{n},1\right\rangle \right]=b<\infty,\,\,\forall s\geq0,
\end{eqnarray*}
since initial conditions have a finite mass. Also note that $\mbox{\ensuremath{\sum_{k\geq0}\left|k-1\right|\nu_{k}<\infty}}$.
Then \eqref{eq:Eg<inf} holds.
\end{proof}
Now we are ready to prove the following lemma.
\begin{lem}
\textup{\label{lem:apply-kurtz}For any $n\geq1$, and every $t>0,$
\begin{align}
  & \mathbb{E}\left[H\left(\xi_{t}^{n},\eta_{t}^{n},\tilde{u}_{0}^{n},\tilde{v}_{0}^{n}\right)\right]-\mathbb{E}\left[H\left(\xi_{0}^{n},\eta_{0}^{n},\tilde{u}_{t}^{n},\tilde{v}_{t}^{n}\right)\right]\nonumber \\
   & \,\,\,\,\,=\mathbb{E}\left[\intop_{0}^{t}\left\{ g\left(\xi_{s}^{n},\eta_{s}^{n},\tilde{u}_{t-s}^{n},\tilde{v}_{t-s}^{n}\right)-h\left(\xi_{s}^{n},\eta_{s}^{n},\tilde{u}_{t-s}^{n},\tilde{v}_{t-s}^{n}\right)\right\} ds\right].\label{eq:apply-kurtz-lemma}
\end{align}
}\end{lem}
\begin{proof}
By Lemmas \ref{lem:g-mart}, \ref{lem:h-mart}, \ref{lem:Eg,Eh<inf}
we can apply Lemma \ref{thm:KURTZ-thm} to function
\[
F_{r,s}^{n}=\mathbb{E}\left[H\left(\xi_{r}^{n},\eta_{r}^{n},\tilde{u}_{s}^{n},\tilde{v}_{s}^{n}\right)\right],
\]
and immediately see that \eqref{eq:apply-kurtz-lemma} holds for almost
every $t>0$. However, again by Lemmas \ref{lem:g-mart}, \ref{lem:h-mart},
\ref{lem:Eg,Eh<inf} one can see that both left-hand and right-hand sides
of \eqref{eq:apply-kurtz-lemma} are continuous in $t$. Hence \eqref{eq:apply-kurtz-lemma}
holds for all $t>0$. 
\end{proof}
Define 
\begin{align}
e(T,n)  = & \mathbb{E}\left[\intop_{0}^{\beta_{n}(T)}\left\{ g\left(\xi_{s}^{n},\eta_{s}^{n},\tilde{u}_{\beta_{n}(T)-s}^{n},\tilde{v}_{\beta_{n}(T)-s}^{n}\right)\right.\right.\label{eq:def_e(T,n)}\\
 & \left. \left.-h\left(\xi_{s}^{n},\eta_{s}^{n},\tilde{u}_{\beta_{n}(T)-s}^{n},\tilde{v}_{\beta_{n}(T)-s}^{n}\right)\right\}ds\right].\nonumber 
\end{align}
To finish the proof of Proposition \ref{lem:limFn} we need 
the following proposition.
\begin{prop}
\label{lem:residue}$e(T,n)\rightarrow0$ as $n\rightarrow\infty$.
\end{prop}
The next subsection is devoted to the proof of the above proposition. 
Now we are ready to complete\\
{\bf  Proof of Proposition \ref{lem:limFn}.}
The proof is immediate by Lemma \ref{lem:apply-kurtz} and Proposition
\ref{lem:residue}.
\qed

\subsection{Proof of Proposition \ref{lem:residue}}\label{sec:techn_prop}

Fix $t>0$. For simplicity, denote $f_{s}=H(\xi_{s}^{n},\eta_{s}^{n},\varphi,\psi)$.
Apply the Taylor series expansion on the exponents inside the sums
on the right-hand side of \eqref{eq:g-def} to get
\begin{multline*}
g(\xi_{s}^{n},\eta_{s}^{n},\varphi,\psi)\\
\begin{aligned}= & f_{s}\left\{ \kappa\sum_{x,y\in\Lambda_{n}}\xi_{s}^{n}(x)p_{xy}^{n}\right.\\
 & \times\left[-\varphi(y)-\psi(y)+\varphi(x)+\psi(x)-i\left(\varphi(y)-\psi(y)-\varphi(x)+\psi(x)\right)\right.\\
 & +\frac{1}{2}\left(-\varphi(y)-\psi(y)+\varphi(x)+\psi(x)-i\left(\varphi(y)-\psi(y)-\varphi(x)+\psi(x)\right)\right)^{2}\\
 & \left.\left.+G^{1,1}(\varphi,\psi,x,y)\right]\right\} \\
 & +f_{s}\left\{ \kappa\sum_{x,y\in\Lambda_{n}}\eta_{s}^{n}(x)p_{xy}^{n}\right.\\
 & \times\left[-\varphi(y)-\psi(y)+\varphi(x)+\psi(x)+i\left(\varphi(y)-\psi(y)-\varphi(x)+\psi(x)\right)\right.\\
 & +\frac{1}{2}\left(-\varphi(y)-\psi(y)+\varphi(x)+\psi(x)+i\left(\varphi(y)-\psi(y)-\varphi(x)+\psi(x)\right)\right)^{2}\\
 & +\left.G^{1,2}(\varphi,\psi,x,y)\right]\\
 & +\gamma\sum_{x\in\Lambda_{n}}\xi_{s}^{n}(x)\eta_{s}^{n}(x)\left[\frac{1}{2}\sigma^{2}\left(\varphi(x)+\psi(x)+i\left(\varphi(x)-\psi(x)\right)\right)^{2}+G^{2,1}(\varphi,\psi,x)\right]\\
 & \left.+\gamma\sum_{x\in\Lambda_{n}}\xi_{s}^{n}(x)\eta_{s}^{n}(x)\left[\frac{1}{2}\sigma^{2}\left(\varphi(x)+\psi(x)-i\left(\varphi(x)-\psi(x)\right)\right)^{2}+G^{2,2}(\varphi,\psi,x)\right]\right\} ,\\
 & \forall s\geq0.
\end{aligned}
\end{multline*}
where  we  used 
our assumption on the branching mechanism:
\[
\sum_{k\geq0}\nu_{k}(k-1)=0\,\,\,\mathrm{and}\,\,\,\sum_{k\geq0}\nu_{k}(k-1)^{2}=\sigma^{2}.
\]
For  the error terms 
 in the Taylor expansion we have the following bounds 
\begin{align}
\nonumber
|G^{1,j}(\varphi,\psi,x,y)| &\leq e^{\varphi(x)+\psi(x)}\left|  -\varphi(y)-\psi(y)+\varphi(x)+\psi(x)-i\left(\varphi(y)-\psi(y)-\varphi(x)+\psi(x)\right)\right|^{3}\,\,\, \\
&\leq C_{\eqref{eq:Tbound1}} 
e^{\varphi(x)+\psi(x)}\left( \varphi(y)^3+\psi(y)^3+\varphi(x)^3+\psi(x)^3\right), \;\;j=1,2, \label{eq:Tbound1}
\end{align} 
\begin{align} 
|G^{2,1}(\varphi,\psi,x)| + |G^{2,2}(\varphi,\psi,x)|  &\leq e^{\varphi(x)+\psi(x)}
\sum_{k\geq0}\nu_{k}|k-1|^3 \left( \left|\varphi(x)+\psi(x)-i\left(\varphi(x)-\psi(x)\right)\right|^3\right.
\nonumber 
\\ & \left. \hspace*{4cm}+\left|\varphi(x)+\psi(x)+i\left(\varphi(x)-\psi(x)\right)\right|^3\right)
\nonumber
\\
\label{eq:Tbound2}
&\leq C_{\eqref{eq:Tbound2}} 
e^{\varphi(x)+\psi(x)}\left(\varphi(x)^3+\psi(x)^3\right),
\end{align} 
where the positive constants $C_{\eqref{eq:Tbound1}}, C_{\eqref{eq:Tbound2}}$ are independent of $\varphi, \psi, x,y$ and  in \eqref{eq:Tbound2} we used the assumption on $\sum_{k\geq0}\nu_{k}k^{3}<\infty$ on the branching mechanism.

We use simple algebra to obtain
\begin{align*}
g(\xi_{s}^{n},\eta_{s}^{n},\varphi,\psi)
=  & f_{s}\left\{ \kappa\sum_{x,y\in\Lambda_{n}}\xi_{s}^{n}(x)p_{xy}^{n}\right.\\
 & \times\left[-\varphi(y)-\psi(y)+\varphi(x)+\psi(x)-i\left(\varphi(y)-\psi(y)-\varphi(x)+\psi(x)\right)\right.\\
 & \left.+2\left(\varphi(x)-\varphi(y)\right)\left(\psi(x)-\psi(y)\right)+i\left((\varphi(x)-\varphi(y))^{2}-(\psi(x)-\psi(y))^{2}\right)\right.\\
 & \left.\left.+G^{1,1}(\varphi,\psi,x,y)\right]\right\} \\
 & +f_{s}\left\{ \kappa\sum_{x,y\in\Lambda_{n}}\eta_{s}^{n}(x)p_{xy}^{n}\right.\\
 & \times\left[-\varphi(y)-\psi(y)+\varphi(x)+\psi(x)+i\left(\varphi(y)-\psi(y)-\varphi(x)+\psi(x)\right)\right.\\
 & +2\left(\varphi(x)-\varphi(y)\right)\left(\psi(x)-\psi(y)\right)-i\left((\varphi(x)-\varphi(y))^{2}-(\psi(x)-\psi(y))^{2}\right)\\
 &\left. +G^{1,2}(\varphi,\psi,x,y)\right]\\
 & \left.+\gamma\sum_{x\in\Lambda_{n}}\xi_{s}^{n}(x)\eta_{s}^{n}(x)\left[4\sigma^{2}\varphi(x)\psi(x)+G^{2,1}(\varphi,\psi,x)+ G^{2,2}(\varphi,\psi,x)\right]\right\} ,\,\,\,
\\ &\hspace*{10cm}\forall s\geq0.
\end{align*}

Let us define
\begin{align}
\label{eq:tildefn}
\tilde f_{T,s}^{n} & = H\left(\xi_{s}^{n},\eta_{s}^{n},\tilde{u}_{\beta_{n}(T)-s}^{n},\tilde{v}_{\beta_{n}(T)-s}^{n}\right) \\
\nonumber 
& = e^{-\left\langle \xi_{s}^{n}+\eta_{s}^{n},\tilde{u}_{\beta_{n}(T)-s}^{n}+\tilde{v}_{\beta_{n}(T)-s}^{n}\right\rangle -i\left\langle \xi_{s}^{n}-\eta_{s}^{n},\tilde{u}_{\beta_{n}(T)-s}^{n}-\tilde{v}_{\beta_{n}(T)-s}^{n}\right\rangle },\; 0\leq s\leq T. 
\end{align}

Now by using the above and Lemmas \ref{lem:h-mart}, \ref{lem:g-mart}
we get (recall that $\tilde{\gamma}=\gamma\sigma^{2}$ and $e(T,n)$
is defined in \eqref{eq:def_e(T,n)}):
\begin{align}
e(T,n)&=e_{\xi,RW}(T,n)+e_{\eta,RW}(T,n)+e_{br}(T,n)\label{eq:e(T,n)=00003De_RW+e_br}
\\
\nonumber
&=: \sum_{j=1}^2 e_{\xi,RW,j}(T,n)+\sum_{j=1}^2 e_{\eta,RW,j}(T,n)+e_{br}(T,n),\label{eq:e(T,n)=00003De_RW+e_br}
\end{align}
where

\begin{eqnarray*}
e_{\xi,RW,1}(T,n) & = & \mathbb{E}\intop_{0}^{\beta_{n}(T)}\tilde f_{T,s}^{n}\left\{ \kappa\sum_{x,y\in\Lambda_{n}}p_{xy}^{n}\xi_{s}^{n}(x)\left(2\left(\tilde{u}_{\beta_{n}(T)-s}^{n}(x)-\tilde{u}_{\beta_{n}(T)-s}^{n}(y)\right)\right.\right.\\
 &  & \times\left(\tilde{v}_{\beta_{n}(T)-s}^{n}(x)-\tilde{v}_{\beta_{n}(T)-s}^{n}(y)\right)\\
 &  & \left.\left.+i\left[\left(\tilde{u}_{\beta_{n}(T)-s}^{n}(x)-\tilde{u}_{\beta_{n}(T)-s}^{n}(y)\right)^{2}-\left(\tilde{v}_{\beta_{n}(T)-s}^{n}(x)-\tilde{v}_{\beta_{n}(T)-s}^{n}(y)\right)^{2}\right] 
\right)\right\} 
 ds,
\\
e_{\xi,RW,2}(T,n) & = & \mathbb{E}\intop_{0}^{\beta_{n}(T)}\tilde f_{T,s}^{n}\left\{ \kappa\sum_{x,y\in\Lambda_{n}}p_{xy}^{n}\xi_{s}^{n}(x)
G^{1,1}(\tilde{u}_{\beta_{n}(T)-s}^{n},\tilde{v}_{\beta_{n}(T)-s}^{n},x,y)\right\} 
 ds,
\\
e_{\eta,RW,1}(T,n) & = & \mathbb{E}\intop_{0}^{\beta_{n}(T)}\tilde f_{T,s}^{n}\left\{ \kappa\sum_{x,y\in\Lambda_{n}}p_{xy}^{n}\eta_{s}^{n}(x)\left(2\left(\tilde{u}_{\beta_{n}(T)-s}^{n}(x)-\tilde{u}_{\beta_{n}(T)-s}^{n}(y)\right)\right.\right.\\
 &  & \times\left(\tilde{v}_{\beta_{n}(T)-s}^{n}(x)-\tilde{v}_{\beta_{n}(T)-s}^{n}(y)\right)\\
 &  & \left.\left.-i\left[\left(\tilde{u}_{\beta_{n}(T)-s}^{n}(x)-\tilde{u}_{\beta_{n}(T)-s}^{n}(y)\right)^{2}-\left(\tilde{v}_{\beta_{n}(T)-s}^{n}(x)-\tilde{v}_{\beta_{n}(T)-s}^{n}(y)\right)^{2}\right] 
\right)\right\} 
 ds,\\
e_{\eta,RW,2}(T,n) & = & \mathbb{E}\intop_{0}^{\beta_{n}(T)}\tilde f_{T,s}^{n}\left\{ \kappa\sum_{x,y\in\Lambda_{n}}p_{xy}^{n}\eta_{s}^{n}(x)
G^{1,2}(\tilde{u}_{\beta_{n}(T)-s}^{n},\tilde{v}_{\beta_{n}(T)-s}^{n},x,y)\right\} 
 ds,\\
e_{br}(T,n) & = & \mathbb{E}\intop_{0}^{\beta_{n}(T)}\tilde f_{T,s}^{n}\sum_{x\in\Lambda_{n}}\gamma \xi_{s}^{n}(x)\eta_{s}^{n}(x)
\left(\sum_{j=1}^2 G^{2,j}(\tilde{u}_{\beta_{n}(T)-s}^{n},\tilde{v}_{\beta_{n}(T)-s}^{n},x)
 \right)
ds.
\end{eqnarray*}

Now we are going to show that indeed $e(T,n)$ vanishes, as $n\rightarrow\infty$.
We start with the following technical lemma that was proved in Lemma 2.1 in \cite{CGSh98a}.
\begin{lem}
\label{lem:help-lemma}
 Denote by
$\left\{ p_{t}^{n}(x,y): t\geq 0, x,y \in \Lambda_n \right\}$ the transition probabilities of the symmetric nearest
neighbor random walk on the domain~$\Lambda_{n}$, and let $\left\{ p_{t}(x,y): t\geq 0, x,y \in\mathbb{Z}^{d} \right\}$ denote the corresponding  transition probabilities 
 on $\mathbb{Z}^{d}$. Let $\left\{ g_{t}(\cdot)\right\} _{t\geq0}$ be the Green function
of the symmetric nearest neighbor random walk on $\mathbb{Z}^{d}$. Then the following
holds.

a) If $t_{n}/n^{2}\rightarrow\infty$ as $n\rightarrow\infty$ , then
\[
\sup_{t\geq t_{n}}\sup_{x,y\in\Lambda_{n}}(2n)^{d}\left|p_{t}^{n}(x,y)-(2n)^{-d}\right|\rightarrow 0.
\]


b) If $d\geq3$, and $T(n)/\left|\Lambda_{n}\right|\rightarrow s\in(0,\infty)$
as $n\rightarrow\infty$, then 
\[
\lim_{n\rightarrow\infty}\intop_{0}^{T(n)}p_{2t}^{n}(x,y)dt=\intop_{0}^{\infty}p_{2t}(x,y)dt+s=\frac{1}{2}g_{\infty}(x-y)+s. 
\]

\end{lem}
 First we state the lemma that gives us an important bound on moments of the processes $u^n, v^n$.
This is where the condition~\eqref{eq:gamsigma_bound} on $\tilde{\gamma} = \gamma\sigma^{2}$ is used.



\begin{lem}
\label{cor:E(u^4)<infty}
Let $d\geq 3$ and $\tilde{\gamma}=\gamma\sigma^{2}<\frac{1}{\sqrt{3^{5}}(\frac12 g_{\infty}(0)+1)}$. 
 Let $\left(u_{t}^{n},v_{t}^{n}\right)_{t\geq0}$
be a solution of \eqref{eq:DP-Lambda_n}, with $Q^{n}$ being a $Q$-matrix
of the nearest neighbor random walk on $\Lambda_{n}$.  Let $\vartheta_1,\vartheta_2\geq 0.$ Assume that
 $u_{0}^{n}(x)=\vartheta_1, v_{0}^{n}(x)=\vartheta_2$ for all $x\in \Lambda_n$. 
Then, 
for any $T\leq 1$, 
\[
\sup_{n\geq 1}\sup_{0\leq t\leq\beta_{n}(T)}\sup_{x\in\Lambda_{n}}\mathbb{E}\left(\left(u_{t}^{n}(x)\right)^{4}\right)<\infty,\,\,\,\mathrm{and}\,\,\,\sup_{n\geq 1}\sup_{0\leq t\leq\beta_{n}(T)}\sup_{x\in\Lambda_{n}}\mathbb{E}\left(\left(v_{t}^{n}(x)\right)^{4}\right)<\infty.
\]
\end{lem}
\begin{proof}
The proof is technical, however follows easily from the proof of Lemma 2.2 in \cite{CGSh98a}.
Since $u^n, v^n$ have constant initial conditions it is easy to see that $\mathbb{E}\left(u_{t}^{n}(x)^p\right)$, $ \mathbb{E}\left(v_{t}^{n}(x)^p\right)$,
$\mathbb{E}\left(u_{t}^{n}(x)^p v_t^n(x)^p\right)$ are constant functions in $x$ for $p>0$. Thus,  denote 
$f^n_t=\mathbb{E}\left(\left(u_{t}^{n}(0)\right)^{4}\right), r^n_t=\mathbb{E}\left(\left(v_{t}^{n}(0)\right)^{4}\right), t\geq 0$.
Then following the argument in the proof of Lemma 2.2 in \cite{CGSh98a} (see pages 175-176 there), one gets that there exists a constant $C>0$ such that 
\begin{align}
\nonumber
 f^n(t) &\leq C \vartheta_1^4 + 3^5 \int_0^t p^n_{2s}(0)\,ds \int_0^t p^n_{2s}(0)\tilde\gamma^2\mathbb{E}\left(\left(u_{t-s}^{n}(0)v_{t-s}^{n}(0)\right)^2\right)\,ds\\
 &\leq C \vartheta_1^4 + 3^5\int_0^t p^n_{2s}(0)\,ds\int_0^t p^n_{2s}(0)\tilde\gamma^2\frac{1}{2}(f^n(t-s) +  r^n(t-s))\,ds,
\end{align}
and similarly 
\begin{align}
\nonumber
 r^n(t) 
 &\leq C \vartheta_2^4 + 3^5\int_0^t p^n_{2s}(0)\,ds\int_0^t p^n_{2s}(0)\tilde\gamma^2\frac{1}{2}(f^n(t-s) +  r^n(t-s))\,ds,
\end{align}
Letting  $J^n(t)=\int_0^t p^n_{2s}(0)\,ds,$
 and 
$$ \bar h^n_t = \sup_{s\leq t}  f^n_s + \sup_{s\leq t}  r^n_s\,,\;t\geq 0,$$
we have 
\begin{align}
\nonumber
 \bar h^n(t) 
 &\leq C\vartheta_1^4 \vartheta_2^4+ 3^5J^n(t)^2 \tilde\gamma^2\bar h^n(t),\; t\geq 0. 
\end{align}
From Lemma~\ref{lem:help-lemma}(b) we get that 
$$
\lim_{n\to\infty} J^n(\beta_n(t))=\frac{1}{2}g_\infty(0)+t\leq \frac12 g_\infty(0)+1,\;\;{\rm for}\; t\leq T\l\leq 1. 
$$
Recalling that $\tilde\gamma<\frac{1}{\sqrt{3^{5}}(\frac12 g_{\infty}(0)+1)}$ we get that 
$$ \limsup_{n\to\infty}  \bar  h^n(\beta_n(T))\leq   \frac{C\vartheta_1^4 \vartheta_2^4}{1-3^5\tilde\gamma^2 ( \frac12 g_\infty(0)+1)^2}<\infty, \;  T\leq 1.  
$$
Since  $\bar h^n(\beta_n(T))<\infty$ for each finite $n$, we are done. 

\end{proof}
From this, we derive the following corollary.
\begin{cor}
\label{cor:Finite second moments}For any $x,y\in\Lambda_{n}$,
\[
\sup_{n}\sup_{t\leq\beta_{n}(T)}\sup_{x,y\in\Lambda_{n}}\mathbb{E}\left(\xi_{t}^{n}(x)\xi_{t}^{n}(y)\right),\,\,\sup_{n}\sup_{t\leq\beta_{n}(T)}\sup_{x,y\in\Lambda_{n}}\mathbb{E}\left(\eta_{t}^{n}(x)\eta_{t}^{n}(y)\right)<\infty.
\]
\end{cor}
\begin{proof}
By Lemmas \ref{lem:Second-moment} and \ref{lem:mixed moment} it
is enough to show that. 
\[
\sup_{n}\sup_{t\leq\beta_{n}(T)}\sup_{x,y\in\Lambda_{n}}g_{t}^{n}(x,y)<\infty,
\]
where $\left\{ g_{t}^{n}(\cdot,\cdot)\right\} _{t\geq0}$ is the Green
function of the symmetric nearest neighbor random walk on $\Lambda_{n}$.
For any $t\geq0$, $x,y\in\Lambda_{n}$, we have
\[
g_{t}^{n}(x,y)\leq g_{\beta_{n}(T)}^{n}(x,y)\leq g_{\beta_{n}(T)}^{n}(0,0).
\]
By Lemma \ref{lem:help-lemma}(b)  $\sup_{n}g_{\beta_{n}(T)}^{n}(0,0)$
is finite, and we are done.\end{proof}
Since, for $x,y\in \Lambda_n$,  $p_{t}^{n}(x,y)$, $g_{t}^{n}(x,y)$ are functions of $x-y$, with some abuse of notation we will sometimes use the notation $p_{t}^{n}(x-y)$, $g_{t}^{n}(x-y)$ 
for  $p_{t}^{n}(x,y)$, $g_{t}^{n}(x,y)$ respectively.

In what follows we always  assume that  $\tilde{\gamma}=\gamma\sigma^{2}<\frac{1}{\sqrt{3^{5}}(\frac12 g_{\infty}(0)+1)}$. 
With Lemma~\ref{cor:E(u^4)<infty} at hand we are ready 
to treat the terms $e_{br}(T,n), e_{\xi,RW,2}(T,n), e_{\eta,RW,2}(T,n)$. 

\begin{lem}
\label{lem:branching-part-vanishes}
\begin{equation}
\label{eq:con_br_err}
e_{br}(T,n)\longrightarrow0,\,\,\,\mathrm{as}\, n\rightarrow\infty,
\end{equation}
\begin{equation}
\label{eq:con_rw_err}
e_{\xi,RW,2}(T,n)\longrightarrow0,\,\,\,\mathrm{and}\,\,\,e_{\eta,RW,2}(T,n)\longrightarrow0,\,\,\,\,\mathrm{as}\, n\rightarrow\infty.
\end{equation}

\end{lem}

\begin{proof}
We will show only~\eqref{eq:con_br_err}, since the proof of~\eqref{eq:con_rw_err} follows along the similar lines. 

\begin{align}
\nonumber
|e_{br}(T,n)| & =  \left|\mathbb{E}\intop_{0}^{\beta_{n}(T)}\tilde f_{T,s}^{n}\sum_{x\in\Lambda_{n}}\gamma \xi_{s}^{n}(x)\eta_{s}^{n}(x)
\left(\sum_{j=1}^2 G^{2,j}(\tilde{u}_{\beta_{n}(T)-s}^{n},\tilde{v}_{\beta_{n}(T)-s}^{n},x)
 \right)
ds\right| 
\\
\nonumber
&\leq
C_{\eqref{eq:Tbound2}}  \mathbb{E}\intop_{0}^{\beta_{n}(T)}|\tilde f_{T,s}^{n}|\sum_{x\in\Lambda_{n}}\gamma \xi_{s}^{n}(x)\eta_{s}^{n}(x)
e^{\tilde{u}_{\beta_{n}(T)-s}^{n}(x)+\tilde{v}_{\beta_{n}(T)-s}^{n}(x)}
\\
\nonumber
&\hspace*{4cm}\times
\left(\tilde{u}_{\beta_{n}(T)-s}^{n}(x)^3+\tilde{v}_{\beta_{n}(T)-s}^{n}(x)^3\right)
ds
\\
\label{eq:ebr_bound1}
&\leq 
C_{\eqref{eq:Tbound2}}\gamma   \mathbb{E}\intop_{0}^{\beta_{n}(T)}\sum_{x\in\Lambda_{n}} \xi_{s}^{n}(x)\eta_{s}^{n}(x)
\left(\tilde{u}_{\beta_{n}(T)-s}^{n}(x)^3+\tilde{v}_{\beta_{n}(T)-s}^{n}(x)^3\right)
ds,
\end{align}
where the first inequality follows by~\eqref{eq:Tbound2} and the second inequality follows by the trivial inequality 
$$|\tilde f_{T,s}^{n} | \xi_{s}^{n}(x)\eta_{s}^{n}(x) e^{\tilde{u}_{\beta_{n}(T)-s}^{n}(x)+\tilde{v}_{\beta_{n}(T)-s}^{n}(x)}\leq  \xi_{s}^{n}(x)\eta_{s}^{n}(x),$$ for all $x\in \Lambda_n$  (recall the definition of $\tilde f_{T,s}^{n}$ in~\eqref{eq:tildefn}). 

Consider the process $(\hat u^{n},\hat v^{n})$ that solves \eqref{eq:DP-Lambda_n}
equations with initial conditions 
\[
\hat u_{0}^{n}(x)=|\Lambda_{n}|\tilde{u}_{0}^{n}(x),\,\,\, \hat v_{0}^{n}(x)=|\Lambda_{n}|\tilde{v}_{0}^{n}(x),\,\,\,\forall x\in\Lambda_{n}.
\]
Then for any $s>0$: 
\[
\hat u_{s}^{n}(x)=|\Lambda_{n}|\tilde{u}_{s}^{n}(x),\,\,\, \hat v_{s}^{n}(x)=|\Lambda_{n}|\tilde{v}_{s}^{n}(x),\,\,\,\forall x\in\Lambda_{n}.
\]
Therefore, 
by the above, ~\eqref{eq:ebr_bound1} and  Fubini's theorem, we get 
\begin{eqnarray*}
|e_{br}(T,n)| 
&\leq& 
|\Lambda_{n}|^{-3}C_{\eqref{eq:Tbound2}}\gamma  \intop_{0}^{\beta_{n}(T)} \mathbb{E}\left[\sum_{x\in\Lambda_{n}} \xi_{s}^{n}(x)\eta_{s}^{n}(x)
\left(\hat{u}_{\beta_{n}(T)-s}^{n}(x)^3+\hat{v}_{\beta_{n}(T)-s}^{n}(x)^3\right)\right]
ds\\
 & \leq & C_{\eqref{eq:Tbound2}}\gamma  T |\Lambda_{n}|^{-1}\theta_{1}\theta_{2}\sup_{s\leq\beta_{n}(T)}\frac{1}{\left|\Lambda_{n}\right|}\mathbb{E}\left[\sum_{x\in\Lambda_{n}}\left(\left(\hat u_{\beta_{n}(T)-s}^{n}(x)\right)^{3}+\left(\hat v_{\beta_{n}(T)-s}^{n}(x)\right)^{3}\right)\right]\\
 & \leq & C_{\eqref{eq:Tbound2}}\gamma  T |\Lambda_{n}|^{-1}\theta_{1}\theta_{2}\sup_{x\in\Lambda_{n}}\sup_{s\leq\beta_{n}(T)}\mathbb{E}\left[\left(\left(\hat u_{\beta_{n}(T)-s}^{n}(x)\right)^{3}+\left(\hat v_{\beta_{n}(T)-s}^{n}(x)\right)^{3}\right)\right],
\end{eqnarray*}
where the second inequality follows by~Corollary~\ref{lem:first-moment}, and the third inequality is trivial. 
With this, to obtain~\eqref{eq:con_br_err},  it is enough to show that
\[
\sup_{n\geq 1}\sup_{x\in\Lambda_{n}}\sup_{s\leq\beta_{n}(T)}\mathbb{E}\left[\left(\left(\hat u_{\beta_{n}(T)-s}^{n}(x)\right)^{3}+\left(\hat v_{\beta_{n}(T)-s}^{n}(x)\right)^{3}\right)\right]<\infty.
\]
 However, this follows from Lemma~\ref{cor:E(u^4)<infty} and Jensen's inequality.  
\end{proof}

Before we begin analyzing  the limiting behavior of $e_{\xi,RW,1}(T,n)$ and
$e_{\eta,RW,1}(T,n)$, we require a technical
lemma whose proof is simple and thus is omitted. 
\begin{lem}
\label{lem:simple-comp}For any $n\in\mathbb{N}$, $r>0$,
\begin{multline*}
\sum_{x_{1},y_{1}\in \Lambda_n}\sum_{x_{2},y_{2}\in \Lambda_n}p_{x_{1},y_{1}}^{n}p_{x_{2},y_{2}}^{n}\left(p_{r}^{n}(x_{1}-x_{2})+p_{r}^{n}(y_{1}-y_{2})+p_{r}^{n}(x_{1}-y_{2})+p_{r}^{n}(y_{1}-x_{2})\right)
= 4\left|\Lambda_{n}\right|.
\end{multline*}

\end{lem}
Now we are ready to derive the limiting behavior of $e_{\xi,RW,1}(T,n)$ and
$e_{\eta,RW,1}(T,n)$.

\begin{lem}
\label{eq:err_rw}
\textup{
\[
\lim_{n\rightarrow\infty}e_{\xi,RW,1}(T,n)=0,\,\,\,\lim_{n\rightarrow\infty}e_{\eta,RW,1}(T,n)=0.
\]
}\end{lem}
\begin{proof}
We will take care of $e_{\xi,RW,1}(T,n)$; the proof for $e_{\eta,RW,1}(T,n)$ is the same.
\begin{multline}
\left|e_{\xi,RW,1}(T,n)\right|\\
\begin{aligned}\,\,\,\,\,\,\,\,\,\,\,\,\leq & C\kappa\mathbb{E}\left(\intop_{0}^{\beta_{n}(T)}\left| \tilde f_{T,s}^{n}\right|\left\{ \left|\sum_{x,y\in\Lambda_{n}}p_{xy}^{n}\xi_{\beta_{n}(T)-s}^{n}(x)\left(\tilde{u}_{s}^{n}(x)-\tilde{u}_{s}^{n}(y)\right)\left(\tilde{v}_{s}^{n}(x)-\tilde{v}_{s}^{n}(y)\right)\right|\right.\right.\\
 & \left.\left.+\left|\sum_{x,y\in\Lambda_{n}}p_{xy}^{n}\xi_{\beta_{n}(T)-s}^{n}(x)\left(\tilde{u}_{s}^{n}(x)-\tilde{u}_{s}^{n}(y)\right)^{2}-\left(\tilde{v}_{s}^{n}(x)-\tilde{v}_{s}^{n}(y)\right)^{2}\right|ds\right\} \right)\\
\,\,\,\,\,\,\,\,\,\,\,\,\leq & C\mathbb{E}\left(\intop_{0}^{\beta_{n}(T)}J_{1}^{n}(s)ds+\intop_{0}^{\beta_{n}(T)}J_{2}^{n}(s)ds\right),
\end{aligned}
\label{eq:J1+J2}
\end{multline}
where
\begin{eqnarray*}
J_{1}^{n}(s) & = & \left|\sum_{x,y\in\Lambda_{n}}p_{xy}^{n}\xi_{\beta_{n}(T)-s}^{n}(x)\left(\tilde{u}_{s}^{n}(x)-\tilde{u}_{s}^{n}(y)\right)\left(\tilde{v}_{s}^{n}(x)-\tilde{v}_{s}^{n}(y)\right)\right|,\\
J_{2}^{n}(s) & = & \left|\sum_{x,y\in\Lambda_{n}}p_{xy}^{n}\xi_{\beta_{n}(T)-s}^{n}(x)\left(\left(\tilde{u}_{s}^{n}(x)-\tilde{u}_{s}^{n}(y)\right)^{2}-\left(\tilde{v}_{s}^{n}(x)-\tilde{v}_{s}^{n}(y)\right)^{2}\right)\right|.
\end{eqnarray*}
Let's bound the expected value of $J_{1}^{n}$:
\begin{eqnarray*}
\mathbb{E}(J_{1}^{n}(s)) & = & \mathbb{E}\left|\sum_{x,y\in\Lambda_{n}}p_{xy}^{n}\xi_{\beta_{n}(T)-s}^{n}(x)\left(\tilde{u}_{s}^{n}(x)-\tilde{u}_{s}^{n}(y)\right)\left(\tilde{v}_{s}^{n}(x)-\tilde{v}_{s}^{n}(y)\right)\right|\\
 & \leq & \sqrt{\mathbb{E}\left[\left(\sum_{x,y\in\Lambda_{n}}p_{xy}^{n}\xi_{\beta_{n}(T)-s}^{n}(x)\left(\tilde{u}_{s}^{n}(x)-\tilde{u}_{s}^{n}(y)\right)\left(\tilde{v}_{s}^{n}(x)-\tilde{v}_{s}^{n}(y)\right)\right)^{2}\right]}.
\end{eqnarray*}
Now we will recall the following representation from Theorem 2.2 in
\cite{DawsonPerkins98} with $\varphi=\delta_{x}$: 
\[
\begin{cases}
\tilde{u}_{t}^{n}(x)=P_{t}^{n}\tilde{u}_{0}^{n}(x)+\sum_{z\in\Lambda_{n}}\intop_{0}^{t}p_{t-s}^{n}(x-z)\sqrt{\tilde{\gamma}\tilde{u}_{s}^{n}(z)\tilde{v}_{s}^{n}(z)}dB_{s}(z), & x\in\Lambda_{n},\\
\tilde{v}_{t}^{n}(x)=P_{t}^{n}\tilde{v}_{0}^{n}(x)+\sum_{z\in\Lambda_{n}}\intop_{0}^{t}p_{t-s}^{n}(x-z)\sqrt{\tilde{\gamma}\tilde{u}_{s}^{n}(z)\tilde{v}_{s}^{n}(z)}dW_{s}(z), & x\in\Lambda_{n},
\end{cases}
\]
to get
\begin{align}
N_{r}^{t}(x,y)  := & P_{t-r}^{n}\tilde{u}_{r}^{n}(x)-P_{t-r}^{n}\tilde{u}_{r}^{n}(y)\label{eq:u(x)-u(y)}\\
  = & P_{t}^{n}\tilde{u}_{0}^{n}(x)-P_{t}^{n}\tilde{u}_{0}^{n}(y)\nonumber \\
   & +\sum_{z\in\Lambda_{n}}\intop_{0}^{r}\left(p_{t-s}^{n}(x-z)-p_{t-s}^{n}(y-z)\right)\sqrt{\tilde{\gamma}\tilde{u}_{s}^{n}(z)\tilde{v}_{s}^{n}(z)}dB_{s}(z),\,\,\, r\leq t,\nonumber 
\end{align}
where the last equality follows from the Chapman-Kolmogorov formula.
Similarly, for $r\leq t$ we get
\begin{align}
M_{r}^{t}(x,y)  := & P_{t-r}^{n}\tilde{v}_{t}^{n}(x)-P_{t-r}^{n}\tilde{v}_{t}^{n}(y)\label{eq:v(x)-v(y)}\\
  = & P_{t}^{n}\tilde{v}_{0}^{n}(x)-P_{t}^{n}\tilde{v}_{0}^{n}(y)\nonumber \\
   & +\sum_{z\in\Lambda_{n}}\intop_{0}^{r}\left(p_{t-s}^{n}(x-z)-p_{t-s}^{n}(y-z)\right)\sqrt{\tilde{\gamma}\tilde{u}_{s}^{n}(z)\tilde{v}_{s}^{n}(z)}dW_{s}(z)\nonumber 
\end{align}
where $\left\{ B_{\cdot}(z)\right\} _{z\in\Lambda_{n}},\left\{ W_{\cdot}(z)\right\} _{z\in\Lambda_{n}}$
are orthogonal Brownian motions.

Note that $\left\{ N_{r}^{t}(x,y)\right\} _{0\leq r\leq t}$ and $\left\{ M_{r}^{t}(x,y)\right\} _{0\leq r\leq t}$
are martingales;
 in addition
\begin{equation}
\begin{array}{c}
\tilde{u}_{t}^{n}(x)-\tilde{u}_{t}^{n}(y)=\left.P_{t-r}^{n}\tilde{u}_{r}^{n}(x)-P_{t-r}^{n}\tilde{u}_{r}^{n}(y)\right|_{r=t}=\left.N_{r}^{t}(x,y)\right|_{r=t},\\
\tilde{v}_{t}^{n}(x)-\tilde{v}_{t}^{n}(y)=\left.P_{t-r}^{n}\tilde{v}_{r}^{n}(x)-P_{t-r}^{n}\tilde{v}_{r}^{n}(y)\right|_{r=t}=\left.M_{r}^{t}(x,y)\right|_{r=t}.
\end{array}\label{eq:u_t(x)-u_t(y)}
\end{equation}
 Then by orthogonality of the Brownian motions $B_{\cdot}(z)$ and
$W_{\cdot}(z)$ for all $\mbox{\ensuremath{z\in\Lambda_{n}}}$, and
the It$\mathrm{\hat{o}}$ formula we get
\begin{multline*}
\sum_{x,y\in\Lambda_{n}}p_{x,y}^{n}\xi_{\beta_{n}(T)-s}^{n}M_{s}^{s}(x,y)N_{s}^{s}(x,y)\\
\begin{aligned}= & \sum_{x,y\in\Lambda_{n}}p_{xy}^{n}\xi_{\beta_{n}(T)-s}^{n}(x)\left(\tilde{u}_{s}^{n}(x)-\tilde{u}_{s}^{n}(y)\right)\left(\tilde{v}_{s}^{n}(x)-\tilde{v}_{s}^{n}(y)\right)\\
= & \sum_{x,y\in\Lambda_{n}}p_{xy}^{n}\xi_{\beta_{n}(T)-s}^{n}(x)\sum_{z\in\Lambda_{n}}\intop_{0}^{s}\left(P_{s-r}^{n}\tilde{u}_{r}^{n}(x)-P_{s-r}^{n}\tilde{u}_{r}^{n}(y)\right)\\
 & \times\left(p_{s-r}^{n}(x-z)-p_{s-r}^{n}(y-z)\right)\sqrt{\tilde\gamma\tilde{u}_{r}^{n}(z)\tilde{v}_{r}^{n}(z)}dW_{r}(z)\\
 & +\sum_{x,y\in\Lambda_{n}}p_{xy}^{n}\xi_{\beta_{n}(T)-s}^{n}(x)\sum_{z\in\Lambda_{n}}\intop_{0}^{s}\left(P_{s-r}^{n}\tilde{v}_{r}^{n}(x)-P_{s-r}^{n}\tilde{v}_{r}^{n}(y)\right)\\
 & \times\left(p_{s-r}^{n}(x-z)-p_{s-r}^{n}(y-z)\right)\sqrt{\tilde\gamma\tilde{u}_{r}^{n}(z)\tilde{v}_{r}^{n}(z)}dB_{r}(z)\\
=: & \sum_{x,y\in\Lambda_{n}}p_{xy}^{n}\xi_{\beta_{n}(T)-s}^{n}(x)\sum_{z\in\Lambda_{n}}\tilde{I}_{1,1}^{n}(s,x,y,z)\\
 & +\sum_{x,y\in\Lambda_{n}}p_{xy}^{n}\xi_{\beta_{n}(T)-s}^{n}(x)\sum_{z\in\Lambda_{n}}\tilde{I}_{1,2}^{n}(s,x,y,z)\\
=:&I_{1,1}^{n}(s)+I_{1,2}^{n}(s).
\end{aligned}
\end{multline*}
Note that
\begin{equation}
\mathbb{E}\left(J_{1}^{n}(s)\right)\leq C\sqrt{\left(\mathbb{E}\left[I_{1,1}^{n}(s)\right)^{2}\right]+\mathbb{E}\left[\left(I_{1,2}^{n}(s)\right)^{2}\right]}.\label{eq:EJ1<=00003D}
\end{equation}
So let us bound $\mathbb{E}\left[\left(I_{1,1}^{n}(s)\right)^{2}\right]$: for all $ s\leq\beta_{n}(T)$, we have 
\begin{align}
\mathbb{E}\left[\left(I_{1,1}^{n}(s)\right)^{2}\right]  = & \sum_{x_{1},y_{1}\in \Lambda_n}\sum_{x_{2},y_{2}\in \Lambda_n}\mathbb{E}\left(\xi_{\beta_{n}(T)-s}^{n}(x_{1})\xi_{\beta_{n}(T)-s}^{n}(x_{2})\right)\label{eq:E(I^n_1,1)^2}\\
\nonumber
 &   \times p_{x_{1},y_{1}}^{n}p_{x_{2},y_{2}}^{n}\sum_{z_{1}\in\Lambda_{n}}\sum_{z_{2}\in\Lambda_{n}}\mathbb{E}\left[\tilde{I}_{1,1}^{n}(s,x_{1},y_{1},z_{1})\tilde{I}_{1,1}^{n}(s,x_{2},y_{2},z_{2})\right].\nonumber
\end{align}
 Note that for $z_{1}\neq z_{2}$, $\tilde{I}_{1,1}^{n}(r,x_{1},y_{1},z_{1})$
and $\tilde{I}_{1,1}^{n}(r,x_{2},y_{2},z_{2})$ are orthogonal square
integrable martingales for $r\leq s$, and hence 
\begin{multline*}
\sum_{z_{1}\in\Lambda_{n}}\sum_{z_{2}\in\Lambda_{n}}\mathbb{E}\left[\tilde{I}_{1,1}^{n}(s,x_{1},y_{1},z_{1})\tilde{I}_{1,1}^{n}(s,x_{2},y_{2},z_{2})\right]\\
\begin{aligned}\,\,\,\,\,\,\,\,\,\,\,\,= & \sum_{z\in\Lambda_{n}}\mathbb{E}\left[\tilde{I}_{1,1}^{n}(s,x_{1},y_{1},z)\tilde{I}_{1,1}^{n}(s,x_{2},y_{2},z)\right]\\
\,\,\,\,\,\,\,\,\,\,\,\,= & \tilde\gamma\sum_{z\in\Lambda_{n}}\mathbb{E}\left[\intop_{0}^{s}\left(P_{s-r}^{n}\tilde{u}_{r}^{n}(x_{1})-P_{s-r}^{n}\tilde{u}_{r}^{n}(y_{1})\right)\left(p_{s-r}^{n}(x_{1}-z)-p_{s-r}^{n}(y_{1}-z)\right)\right.\\
 & \left.\times\left(P_{s-r}^{n}\tilde{u}_{r}^{n}(x_{2})-P_{s-r}^{n}\tilde{u}_{r}^{n}(y_{2})\right)\left(p_{s-r}^{n}(x_{2}-z)-p_{s-r}^{n}(y_{2}-z)\right)\tilde{u}_{r}^{n}(z)\tilde{v}_{r}^{n}(z)dr\right]\\
\,\,\,\,\,\,\,\,\,\,\,\,= & \tilde\gamma\sum_{z\in\Lambda_{n}}\mathbb{E}\left[\intop_{0}^{s}\sum_{z_{1}\in\Lambda_{n}}\left(p_{s-r}^{n}(x_{1}-z_{1})-p_{s-r}^{n}(y_{1}-z_{1})\right)\tilde{u}_{r}^{n}(z_{1})\right.\\
 & \times\sum_{z_{2}\in\Lambda_{n}}\left(p_{s-r}^{n}(x_{2}-z_{2})-p_{s-r}^{n}(y_{2}-z_{2})\right)\tilde{u}_{r}^{n}(z_{2})\\
 & \times\left(p_{s-r}^{n}(x_{1}-z)-p_{s-r}^{n}(y_{1}-z)\right)\left(p_{s-r}^{n}(x_{2}-z)-p_{s-r}^{n}(y_{2}-z)\right)\\
 & \times\left.\tilde{u}_{r}^{n}(z)\tilde{v}_{r}^{n}(z)dr\right]\\
\,\,\,\,\,\,\,\,\,\,\,\, \leq &\tilde\gamma\sum_{z\in \Lambda_n}\intop_{0}^{s}\sum_{z_{1}\in \Lambda_n}\sum_{z_{2}\in \Lambda_n} \hat J_{1,1}(\vec{x},\vec{y},\vec{z},s-r)
\mathbb{E}\left[\tilde{u}_{r}^{n}(z_{1})\tilde{u}_{r}^{n}(z_{2})\tilde{u}_{r}^{n}(z)\tilde{v}_{r}^{n}(z)\right]dr.
\end{aligned}
\end{multline*}
where  $\vec{x}=(x_{1},x_{2}),\vec{y}=(y_{1},y_{2}),\vec{z}=(z_{1},z_{2},z)$ and 
$$
\hat J_{1,1}(\vec{x},\vec{y},\vec{z},s-r)=\Pi_{i=1,2}\left|p_{s-r}^{n}(x_{i}-z_{i})-p_{s-r}^{n}(y_{i}-z_{i})\right|\left|p_{s-r}^{n}(x_{i}-z)-p_{s-r}^{n}(y_{i}-z)\right|.
$$

By Lemma~\ref{cor:E(u^4)<infty} and assumption on the initial
conditions of $(\tilde{u},\tilde{v})$, 
\[
\mathbb{E}\left[\tilde{u}_{r}^{n}(z_{1})\tilde{u}_{r}^{n}(z_{2})\tilde{u}_{r}^{n}(z)\tilde{v}_{r}^{n}(z)\right]
\]
is bounded by $C\left|\Lambda_{n}\right|^{-4}$ uniformly on $z,z_{1},z_{2}\in\Lambda_{n}$,
$r\leq\beta_{n}(T)$ and $n\geq 1$. Therefore,
\begin{multline}
\begin{aligned}
\sum_{z_{1}\in\Lambda_{n}}&\sum_{z_{2}\in\Lambda_{n}}\mathbb{E}\left[\tilde{I}_{1,1}^{n}(s,x_{1},y_{1},z_{1})\tilde{I}_{1,1}^{n}(s,x_{2},y_{2},z_{2})\right]\\
&\leq 
 C\left|\Lambda_{n}\right|^{-4}\sum_{z\in \Lambda_n}\intop_{0}^{s}\sum_{z_{1}\in \Lambda_n}\sum_{z_{2}\in \Lambda_n} \hat J_{1,1}(\vec{x},\vec{y},\vec{z},s-r)dr.
\end{aligned}
\label{eq:sum-sum-*}
\end{multline}
Denote 
\begin{eqnarray*}
\tilde{J}_{1,1}(\vec{x},\vec{y},s-r) & := & \sum_{z\in\Lambda_{n}}\sum_{z_{1}\in\Lambda_{n}}\sum_{z_{2}\in\Lambda_{n}}
 \hat J_{1,1}(\vec{x},\vec{y},\vec{z},s-r).
\end{eqnarray*}
Now we decompose the term on the right-hand side of \eqref{eq:sum-sum-*}
into two terms
\begin{equation}
C\left|\Lambda_{n}\right|^{-4}\intop_{0}^{\left(s-n^{\delta}\right)_{+}}\tilde{J}_{1,1}(\vec{x},\vec{y},s-r)dr+C\left|\Lambda_{n}\right|^{-4}\intop_{\left(s-n^{\delta}\right)_{+}}^{s}\tilde{J}_{1,1}(\vec{x},\vec{y},s-r)dr\label{eq:J11-**}
\end{equation}
 for some $\delta\in(2,d)$.

By Lemma~\ref{lem:help-lemma}(a) 
 we get
\[
\lim_{n\rightarrow\infty}\sup_{t>n^{\delta}}\sup_{z_{1},z_{2}\in\Lambda_{n}}\left|\Lambda_{n}\right|\left|p_{t}^{n}(z_{1},z_{2})-(2n)^{-d}\right|=0,
\]
for any $\delta>2$. This implies that, for any $\delta>2$, there
exists a sequence $a_{n}=a_{n}(\delta)$, such that
\begin{equation}
\sup_{s\geq n^{\delta}}\sup_{w_{1},w_{2},w_{3}\in\Lambda_{n}}\left|p_{s}^{n}(w_{1},w_{2})-p_{s}^{n}(w_{3},w_{2})\right|
\leq\frac{a_{n}}{\left|\Lambda_{n}\right|},
\label{eq:bound-J_1^n-1}
\end{equation}
where $a_{n}\rightarrow0$, as $n\rightarrow\infty$.

By \eqref{eq:bound-J_1^n-1} we immediately get
\begin{equation}
\tilde{J}_{1,1}(\vec{x},\vec{y},s-r)\leq\left|\Lambda_{n}\right|^{3}a_{n}^{4}\left|\Lambda_{n}\right|^{-4}=\left|\Lambda_{n}\right|^{-1}a_{n}^{4},\label{eq:J11-bound}
\end{equation}
for $s>n^{\delta}$ and $r\leq s-n^{\delta}$. Hence for $s\leq \beta_n(T)$,  we get
\begin{align}
\left|\Lambda_{n}\right|^{-4}\intop_{0}^{(s-n^{\delta})_{+}}\tilde{J}_{1,1}(\vec{x},\vec{y},s-r)dr & \leq\left|\Lambda_{n}\right|^{-4} a_{n}^{4}\left|\Lambda_{n}\right|^{-1}\intop_{0}^{(s-n^{\delta})_{+}}1dr 
\leq C\left|\Lambda_{n}\right|^{-4}a_{n}^{4},\label{eq:3tri}
\end{align}
where the last inequality follows since $s\leq \beta_n(T)=\left|\Lambda_{n}\right|T$. 

Let us treat the second term in \eqref{eq:J11-**}. Note that 
\[
\sum_{z_{i}\in\Lambda_{n}}\left|p_{s}^{n}(x_{i}-z_{i})-p_{s}^{n}(y_{i}-z_{i})\right|\leq\sum_{z_{i}\in\Lambda_{n}}p_{s}^{n}(x_{i}-z_{i})+p_{s}^{n}(y_{i}-z_{i})\leq2, \forall i=1,2, s\geq 0.
\]
Also
\begin{multline*}
\sum_{z\in\Lambda_{n}}\left|p_{s-r}^{n}(x_{1}-z)-p_{s-r}^{n}(y_{1}-z)\right|\left|p_{s-r}^{n}(x_{2}-z)-p_{s-r}^{n}(y_{2}-z)\right|\\
\begin{aligned}\,\,\,\,\,\, & \leq\sum_{z\in\Lambda_{n}}  \left(p_{s-r}^{n}(x_{1}-z)+p_{s-r}^{n}(y_{1}-z)\right)\left(p_{s-r}^{n}(x_{2}-z)+p_{s-r}^{n}(y_{2}-z)\right)\\
 & =p_{2(s-r)}^{n}(x_{1}-x_{2})+p_{2(s-r)}^{n}(y_{1}-y_{2})+p_{2(s-r)}^{n}(x_{1}-y_{2})+p_{2(s-r)}^{n}(y_{1}-x_{2}),
\end{aligned}
\end{multline*}
where the last equality follows from the Chapman-Kolmogorov formula.
Then
\begin{multline}
C\left|\Lambda_{n}\right|^{-4}\intop_{(s-n^{\delta})_+}^{s}\tilde{J}_{1,1}(\vec{x},\vec{y},s-\
r)dr\\
\,\,\,\,\leq 	
 C\left|\Lambda_{n}\right|^{-4}\intop_{0}^{n^{\delta}}\left(p_{2r}^{n}(x_{1}-x_{2})+p_{2r}^{n}(y_{1}-y_{2})
+p_{2r}^{n}(x_{1}-y_{2})+p_{2r}^{n}(y_{1}-x_{2})\right)dr. 
\label{eq:4tri}
\end{multline}
By \eqref{eq:sum-sum-*}, \eqref{eq:J11-**}, \eqref{eq:3tri} and
\eqref{eq:4tri} we get
\begin{multline*}
\sum_{z_{1}\in\Lambda_{n}}\sum_{z_{2}\in\Lambda_{n}}\mathbb{E}\left[\tilde{I}_{1,1}^{n}(s,x_{1},y_{1},z_{1})\tilde{I}_{1,1}^{n}(s,x_{2},y_{2},z_{2})\right]\\
\begin{aligned}\,\,\,\,\,\leq & C\left|\Lambda_{n}\right|^{-4}\left(a_{n}^{4}+\intop_{0}^{n^{\delta}}\left(p_{2r}^{n}(x_{1}-x_{2})+p_{2r}^{n}(y_{1}-y_{2})
+p_{2r}^{n}(x_{1}-y_{2})+p_{2r}^{n}(y_{1}-x_{2})\right)dr\right).
\end{aligned}
\end{multline*}
Use the above inequality, \eqref{eq:E(I^n_1,1)^2} and also Corollary
\ref{cor:Finite second moments} and Lemma \ref{lem:simple-comp}
to get
\begin{align}
\mathbb{E}\left[\left(I_{1,1}^{n}(s)\right)^{2}\right] & \leq  C\left(\sum_{x_{1},y_{1}\in \Lambda_n}\sum_{x_{2},y_{2}\in \Lambda_n}p_{x_{1},y_{1}}^{n}p_{x_{2},y_{2}}^{n}\left|\Lambda_{n}\right|^{-4}a_{n}^{4}+\left|\Lambda_{n}\right|^{-4}n^{\delta}\left|\Lambda_{n}\right|\right)\nonumber \\
 & \leq  C\left(\left|\Lambda_{n}\right|^{-2}a_{n}^{4}+\left|\Lambda_{n}\right|^{-3}n^{\delta}\right).\label{eq:E(I^n_11)-bound}
\end{align}
In the same way we handle $I_{1,2}^{n}(s)$ and get
\begin{align}
\mathbb{E}\left[\left(I_{1,2}^{n}(s)\right)^{2}\right] & \leq  C\left(\sum_{x_{1},y_{1}\in \Lambda_n}\sum_{x_{2},y_{2}\in \Lambda_n}p_{x_{1},y_{1}}^{n}p_{x_{2},y_{2}}^{n}\left|\Lambda_{n}\right|^{-4}a_{n}^{4}+\left|\Lambda_{n}\right|^{-4}n^{\delta}\left|\Lambda_{n}\right|\right)\nonumber \\
 & \leq  C\left(\left|\Lambda_{n}\right|^{-2}a_{n}^{4}+\left|\Lambda_{n}\right|^{-3}n^{\delta}\right).\label{eq:E(I^n_12)-bound}
\end{align}
By \eqref{eq:E(I^n_11)-bound}, \eqref{eq:E(I^n_12)-bound} and \eqref{eq:EJ1<=00003D},
we have
\begin{eqnarray*}
\mathbb{E}\left[J_{1}^{n}(s)\right] & \leq & \sqrt{\mathbb{E}\left[\left(I_{1,1}^{n}(s)\right)^{2}+\left(I_{1,2}^{n}(s)\right)^{2}\right]}\\
 & \leq & C\sqrt{\left|\Lambda_{n}\right|^{-2}a_{n}^{4}+\left|\Lambda_{n}\right|^{-3}n^{\delta}}\\
 & \leq & C\left|\Lambda_{n}\right|^{-1}a_{n}^{2}+C\left|\Lambda_{n}\right|^{-3/2}n^{\delta/2}.
\end{eqnarray*}
Thus
\begin{align}
\intop_{0}^{\beta_{n}(T)}\mathbb{E}\left(J_{1}^{n}(s)\right)ds & \leq  C\intop_{0}^{\beta_{n}(T)}\left(\left|\Lambda_{n}\right|^{-1}a_{n}^{2}+\left|\Lambda_{n}\right|^{-3/2}n^{\delta/2}\right)ds\label{eq:int-J1}\\
 & \leq  C\left|\Lambda_{n}\right|\left(\left|\Lambda_{n}\right|^{-1}a_{n}^{2}+\left|\Lambda_{n}\right|^{-3/2}n^{\delta/2}\right)\nonumber \\
 & \leq  Ca_{n}^{2}+C\left(\frac{n^{\delta}}{\left|\Lambda_{n}\right|}\right)^{1/2}\rightarrow0,\,\,\,\,\mathrm{as}\,\, n\rightarrow\infty,\nonumber 
\end{align}
where the last convergence holds since $\delta<d$ and $\left|\Lambda_{n}\right|=(2n+1)^{d}$.

Now we are ready to treat $J_{2}^{n}(s)$ in a similar way. 

By the It$\hat{\mathrm{o}}$ formula,
\[
\left(M_{r}^{s}(x,y)\right)^{2}=\intop_{0}^{s}M_{r}^{s}(x,y)dM_{r}^{s}(x,y)+\left\langle M_{\cdot}^{s}(x,y)\right\rangle _{s}
\]
and
\[
\left(N_{r}^{s}(x,y)\right)^{2}=\intop_{0}^{s}N_{r}^{s}(x,y)dN_{r}^{s}(x,y)+\left\langle N_{\cdot}^{s}(x,y)\right\rangle _{s}.
\]
Note that $\left\langle M_{\cdot}^{t}(x,y)\right\rangle _{t}=\left\langle N_{\cdot}^{t}(x,y)\right\rangle _{t}$,
and recall \eqref{eq:u_t(x)-u_t(y)}; therefore
\[
J_{2}^{n}(s)=\sum_{x,y\in\Lambda_{n}}\frac{1}{2}p_{x,y}^{n}\xi_{\beta_{n}(T)-s}^{n}\left[\intop_{0}^{s}M_{r}^{s}(x,y)dM_{r}^{s}(x,y)-\intop_{0}^{s}N_{r}^{s}(x,y)dN_{r}^{s}(x,y)\right].
\]
If we follow the steps of computations for $J_{1}^{n}(s)$, we get
that $$\lim_{n\rightarrow\infty} \mathbb{E}\left( \intop_{0}^{\beta_{n}(T)} J_{2}^{n}(s)\,ds\right)=0.$$

Use this, \eqref{eq:int-J1} and \eqref{eq:J1+J2} to finish the proof.

\end{proof}

{\bf Proof of Proposition \ref{lem:residue}.}\,\, Proposition \ref{lem:residue} follows immediately from~\eqref{eq:e(T,n)=00003De_RW+e_br}, and  Lemmas~ \ref{lem:branching-part-vanishes}, \ref{eq:err_rw}. 
\qed


\bibliographystyle{plain}
\phantomsection\addcontentsline{toc}{section}{\bibname}



\end{document}